\documentclass{amsart}

\input{psfig}

\usepackage{graphicx}
\usepackage{amssymb,mathrsfs}
\usepackage{subfigure}
\usepackage[usenames]{color}
\usepackage{enumerate}        

\newcommand {\bng}{B_n\Gamma}

\newcommand {\ucng}{U\mathcal{C}^n\Gamma}

\newcommand {\dng}{\mathcal{D}^n\Gamma}
\newcommand {\udng}{U\mathcal{D}^n\Gamma}
\newcommand {\bn}[1]{B_n}
\newcommand {\ud}[2]{U\mathcal{D}^{#1} #2}

\newtheorem{theorem}{Theorem}[section]
\newtheorem{lemma}[theorem]{Lemma}

\newtheorem{assumption}[theorem]{Assumption}
\newtheorem{proposition}[theorem]{Proposition}

\newtheorem{conjecture}[theorem]{Conjecture}
\theoremstyle{definition}
\newtheorem{definition}[theorem]{Definition}
\newtheorem{example}[theorem]{Example}

\newtheorem{note}[theorem]{Note}

\input{psfig}

\begin{document}

\title[Presentations of graph braid groups]{Presentations of graph 
braid groups}
\author[D.Farley]{Daniel Farley}
      \address{Department of Mathematics \\
               Miami University\\
               Oxford, OH 45056\\
               http://www.users.muohio.edu/farleyds/}
      \email{farleyds@muohio.edu}
\author[L.Sabalka]{Lucas Sabalka}
      \address{Department of Mathematical Sciences\\
               Binghamton University\\
               Binghamton, NY  13902\\
               http://www.math.binghamton.edu/sabalka}
      \email{sabalka@math.binghamton.edu}

\begin{abstract}

Let $\Gamma$ be a graph.  The (unlabeled) configuration space $\ucng$ of $n$ points on $\Gamma$ is the space of $n$-element subsets of $\Gamma$. The $n$-strand braid group of $\Gamma$, denoted $\bng$, is the fundamental group of $\ucng$.

This paper extends the methods and results of \cite{FarleySabalka1}.   Here we compute presentations for $B_{n}\Gamma$, where $n$ is an arbitrary natural number and $\Gamma$ is an arbitrary finite connected graph.  Particular attention is paid to the case $n = 2$, and many examples are given.

\end{abstract}

                                                                                

\maketitle

\section{Introduction}\label{sec:intro}

Given a graph $\Gamma$, the \emph{unlabeled configuration space} $\ucng$ of $n$ points on $\Gamma$ is the space of $n$-element subsets of distinct points in $\Gamma$. The \emph{$n$-strand braid group of $\Gamma$}, denoted $\bng$, is the fundamental group of $\ucng$.

Graph braid groups are of interest because of their connections with classical braid groups and right-angled Artin groups \cite{CrispWiest,Sabalka,FarleySabalka2a}, and their connections with robotics and mechanical engineering.  Graph braid groups can, for instance, model the motions of robots moving about a factory floor \cite{Ghrist,Farber1,Farber2}, or the motions of microscopic balls of liquid on a nano-scale electronic circuit \cite{GhristPeterson}.

Various properties of graph braid groups have been established.  Ghrist showed in \cite{Ghrist} that the spaces $\ucng$ are $K( B_n \Gamma, 1)$s.  Abrams \cite{Abrams} showed that graph braid groups are fundamental groups of locally CAT(0) cubical complexes, and so for instance have solvable word and conjugacy problems \cite{BridsonHaefliger}.  Crisp and Wiest \cite{CrispWiest} showed that any graph braid group embeds in some right-angled Artin group, so graph braid groups are linear, bi-orderable, and residually finite.  For more information on what is known about graph braid groups, see for instance
\cite{Sabalka4}.

This paper continues a project begun in \cite{FarleySabalka1}.  In \cite{FarleySabalka1}, we used a discrete version of Morse theory (due to Forman \cite{Forman}) to simplify the configuration spaces $\ucng$ within their homotopy types. We were able to compute presentations for all braid groups $B_n T$, where $T$ is a tree; that is, for all \emph{tree braid groups} (\cite{FarleySabalka1}, Theorem 5.3).  Our methods also allowed us, in principle, to compute a Morse presentation for any graph braid group.  

Here we describe how to compute presentations for all graph braid groups $B_{n}\Gamma$, where $n$ is an arbitrary natural number and $\Gamma$ is a finite connected graph.  As in \cite{FarleySabalka1}, the generators in our presentations correspond to critical $1$-cells, and the relators
correspond to critical $2$-cells.  (In both cases, ``critical'' is meant in the sense of Morse theory.)  Theorem \ref{thm:biggie} describes the general form of a relator in $B_{n}\Gamma$ using ``costs", which are certain words in the group generators.  Propositions \ref{prop:vanishing} and \ref{prop:cost} describe exactly how to compute these costs.  Theorem \ref{thm:biggie} and Propositions \ref{prop:vanishing} and \ref{prop:cost} therefore completely describe presentations for all graph braid groups.  The resulting presentations are particularly simple when $n=2$.  Note \ref{note:important} describes a procedure for computing the relators in a 2-strand graph braid group, $B_{2}\Gamma$.      

The paper is organized as follows.  Section \ref{sec:background} contains a basic introduction to the version of discrete Morse theory that will be used in the rest of the paper.  Section \ref{sec:DGVF} describes a Morse matching on every `discretized' configuration space $\ud{n}{\Gamma}$.  Section \ref{sec:main} shows how to use the ideas of the previous sections in order to compute presentations of graph braid groups.  Section \ref{sec:morse2pres} collects some results about the case of two strands, and, in particular, contains the important Note \ref{note:important} 
(as described above).  Finally, Section \ref{sec:threeandmore} contains computations of presentations of $B_{n}\Gamma$, where $n=2$ or $3$ and $\Gamma$ is a balloon graph.
    
\section{Background on Discrete Morse Theory}\label{sec:background}

\subsection{A motivating example} \label{sec:motivatingexample}

\begin{example}

Consider the `star' tree $Y_k$, $k \geq 3$ -- the tree with exactly $1$ vertex of degree $k$ and $k$ vertices of degree $1$.  Then $B_2Y_k$ is a free group of rank ${{k-1}\choose{2}}$ (cf. \cite{Ghrist,ConnollyDoig}).  Ghrist described the entire configuration space $C_3 := U\mathcal{C}^2(Y_3)$ as a union of squares and triangles.  The configuration space $C_3$ is reproduced here on the left in Figure \ref{fig:B_2Y_k}.  The three (distorted) squares correspond to configurations in which the two strands occupy different branches of $Y_2$.  There are three squares since there are three ways of choosing two edges from a set of three.  The triangular flaps are copies of the configuration space $U\mathcal{C}^{2}(I)$, where $I$ is the unit interval.  These flaps correspond to configurations in which both strands occupy the same closed edge.

\begin{figure}[!h]
\begin{center}
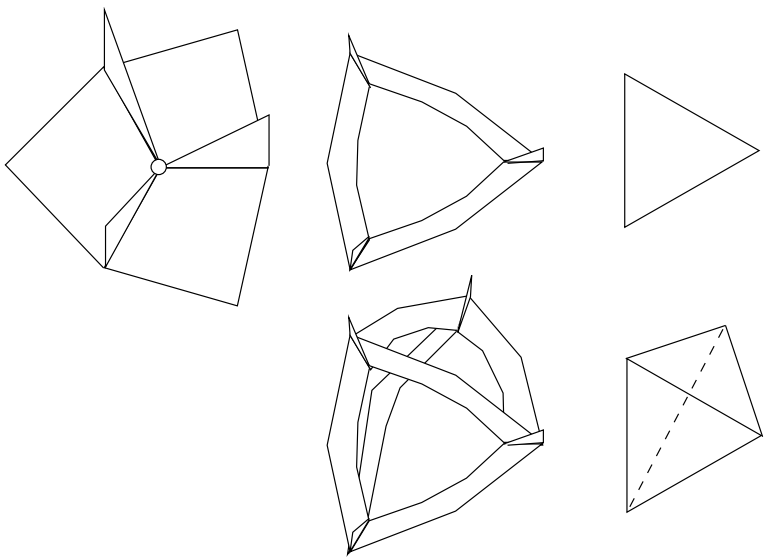

\caption{On the left is the configuration space $C_3$ of $2$ strands on the tree $Y_n$.  The upper line shows that $C_3$ is homotopy equivalent to the $1$-skeleton of the $2$-simplex (i.e. a triangle).  The lower pictures show the space $C_4$ and the $1$-skeleton of the $3$-simplex, which are homotopy equivalent.}
\label{fig:B_2Y_k}
\end{center}
\end{figure}

Adjacent to our reproduction of Ghrist's $C_3$ is a space homeomorphic to $C_3$, where the central missing point has been expanded into a triangle.  We use this picture to show the configuration spaces $C_k := U\mathcal{C}^2(Y_k)$ where $k \geq 4$.

We begin with $k = 4$.  It is straightforward to see that we need to assemble the configuration space $C_4$ out of $4$ copies of $C_3$, for a total of ${4 \choose 2} = 6$ squares (the number of 
ways of choosing $2$ edges for the two strands) and $4$ triangles (the number of ways of choosing $1$ edge for both strands).  In $C_3$, the squares and triangles were joined together along a common edge, with cross-section a copy of $Y_3$.  In $C_4$, the squares and triangles will also come together along a common edge, but with $Y_4$ as cross-section.  It is also straightforward to see that the desired shape for $C_4$ should resemble the $1$-skeleton of the $3$-simplex.  Indeed, as $C_3$ deformation retracts onto the $1$-skeleton of the $2$-simplex (that is, a triangle), so $C_4$ deformation retracts onto the $1$-skeleton of the $3$-simplex.

These arguments generalize as follows.  We build the configuration space $C_k$ from the $1$-skeleton of a $(k-1)$-simplex.  For each vertex $v$ of the simplex, $C_k$ has a half-open edge (that is, one endpoint is present and one is not), labeled $c_v$.  For each $1$-cell $e$ of the simplex, $C_k$ has a square $c_e$ which is closed except that it is missing a single corner.  A square $c_e$ is attached to the edges $c_{\iota e}$ and $c_{\tau e}$ corresponding to the endpoints $\iota e$ and $\tau e$ of $e$, along \emph{adjacent} faces of $c_e$, so that the missing endpoints of $c_{\iota e}$ and $c_{\tau e}$ match the missing corner of $c_e$.  Finally, to each edge $c_v$ of $C_k$ there is attached a right triangle, which is closed except that it is missing a single corner adjacent to the hypotenuse.  The triangle is glued to $c_v$ along the side from the right angle to the missing corner, so that the missing corner lines up with the missing vertex of $c_v$.  

The deformation retraction onto the $1$-skeleton of the $k$-simplex is then clear, reinforcing the calculation that $B_2Y_k$ is a free group of rank ${{k-1}\choose{2}}$.

Even for very simple graphs like the graphs $Y_k$, the full configuration space can quickly contain much more information than is necessary to understand the behavior of the corresponding graph braid group.  As this example shows, there is a lot of extraneous information contained in the full configuration space, and the configuration spaces can quickly become difficult to visualize.  There are much smaller homotopy equivalent spaces, like the $1$-skeleta of simplices indicated here, which carry all of the homotopy data needed to compute presentations of the graph braid groups.  Indeed, even the $1$-skeleta of simplices, although $1$-dimensional instead of $2$-dimensional, are still too large and themselves have the homotopy types of bouquets of circles.  

\end{example}

As in this example, we wish to simplify configuration spaces of graphs down to smaller, more manageable spaces.  We use Forman's discrete Morse theory to formalize a method of simplifying these configuration spaces.  Discrete Morse theory deals with CW complexes, but as we will see in Section \ref{sec:DGVF}, this is not a problem for us.

\subsection{Discrete Morse theory:  definitions} \label{sec:morsetheory} In this subsection, we collect some basic definitions (including the definition of a Morse matching) from \cite{FarleySabalka1} (see also \cite{Brown} and \cite{Forman}, which were the original sources for these ideas).
                                                                                
Let $X$ be a finite regular CW complex. Let $K$ denote the set of open cells of $X$.  Let $K_{p}$ be the set of open $p$-cells of $X$.  For open cells $\sigma$ and $\tau$ in $X$, we write $\sigma < \tau$ if $\sigma \neq \tau$ and $\sigma \subseteq \overline{\tau}$, where $\overline{\tau}$ is the closure of $\tau$, and $\sigma \leq \tau$ if $\sigma < \tau$ or $\sigma = \tau$.

A \emph{partial function} from a set $A$ to a set $B$ is a function defined on a subset of $A$, and having $B$ as its target.  A \emph{discrete vector field} $W$ on $X$ is a sequence of partial functions $W_{i}: K_{i} \rightarrow K_{i+1}$ such that:
                                                                                
\begin{enumerate}
\item Each $W_{i}$ is injective;
\item if $W_{i} (\sigma ) = \tau$, then $\sigma < \tau$;
\item $\operatorname{image} \left( W_{i} \right) \cap \operatorname{domain} \left( W_{i+1} \right) =
\emptyset$.
\end{enumerate}

Let $W$ be a discrete vector field on $X$.  A \emph{$W$-path of dimension $p$} is a sequence of $p$-cells $\sigma_{0}, \sigma_{1}, \ldots, \sigma_{r}$ such that if $W( \sigma_{i} )$ is undefined, then $\sigma_{i+1} = \sigma_{i}$; otherwise $\sigma_{i+1} \neq \sigma_{i}$ and $\sigma_{i+1} < W( \sigma_{i})$.  The $W$-path is \emph{closed} if $\sigma_{r} = \sigma_{0}$, and \emph{non-stationary} if $\sigma_{1} \neq \sigma_{0}$.  A discrete vector field $W$ is a \emph{Morse matching} if $W$ has no non-stationary closed paths.

If $W$ is a Morse matching, then a cell $\sigma \in K$ is \emph{redundant} if it is in the domain of $W$, \emph{collapsible} if it is in the image of $W$, and \emph{critical} otherwise.  Note that any two of these categories are mutually exclusive by condition (3) in the definition of discrete vector field.

The ideas ``discrete Morse function" and ``Morse matching" are largely equivalent, in a sense that is made precise in \cite{Forman}, pg. 131.  In practice, we will always use Morse matchings instead of discrete Morse functions in this paper (as we also did in \cite{FarleySabalka1}). A Morse matching is sometimes referred to as a ``discrete gradient vector field'' in the literature; in particular, we used this more cumbersome terminology in \cite{FarleySabalka1}.

\subsection{Discrete Morse theory and fundamental group} \label{sec:fundgp} 

Section 2 of \cite{FarleySabalka1} described how to compute the fundamental group of a finite regular CW complex $X$ using a Morse matching $W$ defined on $X$.  This is done as follows.

Choose an orientation for $1$-cells in $X$, and let $A$ denote the set of all oriented $1$-cells and their inverses, treated as an alphabet. For every $1$-cell in $X$, we want to rewrite the $1$-cell in terms of $1$-cells that are critical under $W$.   Consider a word $w$ in $A^*$.  For any words $w_1$, $w_2$, and $w_3$ in $A^*$ and any single oriented $1$-cell $e \in A$, we define
the following moves:

\begin{enumerate}
\item (free cancellation) If $w = w_1ee^{-1}w_2$ or $w = w_1e^{-1}ew_2$, write $w \rightarrow w_1w_2$.

\item (collapsing) If $w = w_1ew_2$ and $e$ is a collapsible $1$-cell, write $w \rightarrow w_1w_2$.

\item (simple homotopy) If $w = w_1ew_2$ or $w_1e^{-1}w_2$, and $ew_3$ is a boundary word for a $2$-cell $c$ satisfying $W(e) = c$, write $w \rightarrow w_1w_3^{-1}w_2$ or $w \rightarrow w_1w_3w_2$, respectively.
\end{enumerate}
Let $\dot{\rightarrow}$ denote the reflexive transitive closure of $\rightarrow$, as in \cite{FarleySabalka1}.

Proposition 2.4 of \cite{FarleySabalka1} implies that any sequence
	$$w \rightarrow w_1 \rightarrow w_2 \rightarrow \dots$$
must terminate in a \emph{reduced} word, i.e.,  one that is not the source of any arrow.  Furthermore, this reduced word is uniquely determined by $w$, not the particular sequence of arrows.  We denote this reduced word $M^{\infty}(w)$.  The free cancellation property shows $M^{\infty}(w)$ is freely reduced, and the collapsing and simple homotopy properties guarantee $M^{\infty}(w)$ consists entirely of critical $1$-cells. 

\begin{theorem} \label{thm:group} \emph{(}\cite{FarleySabalka1}\emph{)}
Let $X$ be a finite regular connected CW complex and let $W$ be a Morse matching with only one critical $0$-cell.  Then:
	$$ \pi_{1}(X) \cong \langle \Sigma \mid \mathcal{R} \rangle,$$
where $\Sigma$ is the set of positively-oriented critical $1$-cells, and $\mathcal{R}$ is the set of all words $M^{\infty}(w)$ in $\Sigma$, where $w$ runs over all boundaries of critical $2$-cells.
\end{theorem}

\begin{proof}
This theorem is almost exactly Theorem 2.5 of \cite{FarleySabalka1}, with the simplifying assumption that there is only $1$ critical $0$-cell.  Under this assumption, Propositions 2.2(2) and 2.3(5) of \cite{FarleySabalka1} imply that the union of all collapsible edges is a maximal tree in $X$.  Thus, the tree $T$ named in Theorem 2.5 of \cite{FarleySabalka1} contains no critical $1$-cells.  The statement of the theorem given here then follows.
\end{proof}

\begin{definition}
For a given $X$ and $W$ as in Theorem \ref{thm:group}, we call the presentation of Theorem \ref{thm:group} a \emph{Morse presentation} for $\pi_1X$.
\end{definition}

\section{Morse Matchings for Configuration Spaces of Graphs}\label{sec:DGVF}

\subsection{The complex $UD^n\Gamma$} 

The configuration spaces $\ucng$ were used to define graph braid groups, but unfortunately these have no natural CW-complex structure.  We need slightly different spaces, which we describe here.

Let $\Gamma$ be a finite connected graph and $n$ a natural number.  Let $\Delta'$ denote the union of those open cells of $\prod^n \Gamma$ (equipped with the product cell structure) whose closures intersect the fat diagonal $\Delta = \{(x_1,\dots,x_n)|x_i=x_j\hbox{ for some }i\neq j\}$. Let $\dng$ denote the space $\prod^n \Gamma - \Delta'$, called the \emph{discretized configuration space} of $n$ points on $\Gamma$. Note that $\dng$ inherits a CW complex structure from the Cartesian product.  A cell in $\dng$ has the form $c_1 \times \dots \times c_n$ such that:  each $c_i$ is either a vertex or the interior of an edge, and the closures of the $c_i$ are mutually disjoint.

Let $\udng$ denote the quotient of $\dng$ by the action of the symmetric group $S_n$ which permutes the coordinates. Thus, an open cell in $\udng$ has the form $\{c_1, \dots, c_n\}$ such that:  each $c_i$ is either a vertex or the interior of an edge, and the closures are mutually disjoint.  The set notation is used to indicate that order does not matter.
                                                                                
Under most circumstances, the labeled or unlabeled configuration space of $\Gamma$ is homotopy equivalent to the corresponding discretized version.  Specifically:
                                                                                
\begin{theorem} \cite{Hu, Abrams, PrueScrimshaw}  \label{thm:Abrams}
For any integer $n\geq 2$ and any graph $\Gamma$ with at least $n$ vertices, the labeled (unlabeled) configuration space of $n$ points on $\Gamma$ strong deformation retracts onto $\dng$ ($\udng$) if
\begin{enumerate}
\item each path between distinct vertices of degree not equal to $2$ passes through at least $n-1$ edges; and

\item each path from a vertex to itself which is not null-homotopic in $\Gamma$ passes through at least $n+1$ edges.
\end{enumerate}
\end{theorem}
This theorem was orignally proved by Hu \cite{Hu} for the case $n = 2$.  Arbitrary $n$ was done by Abrams \cite{Abrams}, but with a slightly weaker first condition.  Full generality is proven by Prue and Scrimshaw \cite{PrueScrimshaw}, who also address a minor flaw in the proof of Abrams.

A graph $\Gamma$ satisfying the conditions of this theorem for a given $n$ is called \emph{sufficiently subdivided} for this $n$.  It is clear that, for any fixed $n$, every graph is homeomorphic to a sufficiently subdivided graph.

Throughout the rest of the paper, we work exclusively with the space $\udng$ where $\Gamma$ is sufficiently subdivided for $n$.  Also from now on, ``edge'' and ``cell'' will refer to closed objects.

\subsection{Some definitions related to the Morse matchings}

This section contains definitions which will help us when we define Morse matchings on the complexes $\ud{n}{\Gamma}$ in Subsection \ref{sec:Morsematching}.

As usual, let $\Gamma$ be a finite connected graph.  For a vertex $v$ of $\Gamma$, let $deg_{\Gamma}(v)$ or just $deg(v)$ denote the degree of $v$.  If $deg(v) \geq 3$, then $v$ is called an \emph{essential} vertex.

We define an order on the vertices of $\Gamma$ as follows.  Choose a maximal tree $T$ in $\Gamma$. Edges outside of $T$ are called \emph{deleted edges}.  Pick a vertex $\ast$ of degree $1$ in $T$ to be the root of $T$. Choose an embedding of the tree $T$ into the plane.  Begin at the basepoint $\ast$ and walk along the tree, following the leftmost branch at any given intersection, and consecutively number the vertices by the order in which they are first encountered.  When a vertex of degree one is reached, turn around.  The vertex adjacent to $\ast$ is assigned the number $1$, and $\ast$ is (in effect) numbered $0$.  Note that this numbering depends only on the choice of $\ast$ and the embedding of the tree.  For two vertices $v$ and $w$, we write $v < w$ if $v$ receives a smaller number than $w$ (i.e., if $v$ is encountered before $w$ in a clockwise traversal of $T$ from $\ast$).  We let $(v,w)$ denote the open interval in the order on vertices:  
	$$(v,w) = \{ u \in \Gamma^{0} \mid v < u < w \}.$$

When we write $[v,w]$, we mean the geodesic connecting $v$ to $w$ within $T$.  The \emph{meet} of two vertices $v$ and $w$ in $T$, denoted $v \wedge w$ is the greatest vertex in the intersection $[\ast, v] \cap [\ast, w]$.

For a given edge $e$ of $\Gamma$, let $\iota e$ and $\tau e$ denote the endpoints of $e$.  We orient each edge to go from $\iota e$ to $\tau e$, and so that $\iota e > \tau e$.  Thus, if $e \subseteq T$, the geodesic segment $[ \iota e, \ast ]$ in $T$ must pass through $\tau e$. For a vertex $v$, let $e(v)$ be the unique edge in $T$ satisfying $\iota (e(v)) = v$.  

The choice and embedding of $T$ allow us to refer to directions in $T$, as follows.  If $v$ is a vertex in the tree $T$, we say that two vertices $v_1$ and $v_2$ lie in the same \emph{direction} from $v$ if the geodesics $[ v , v_1 ] , [v , v_2 ] \subseteq T$ start with the same edge.  Thus, there are $deg_T(v)$ directions from a vertex $v$.  We number these directions $0, 1 , 2, \ldots, deg_T(v)-1$, beginning with the direction represented by $[v, \ast]$, numbered $0$, and proceeding in clockwise order.  A vertex is considered to be in direction $0$ from itself.  We let $d(v_{1}, v_{2})$ denote the direction from $v_{1}$ to $v_{2}$ (which is an integer).  When referring to a \emph{direction} from a vertex to an edge, we are referring to the direction from the given vertex to the initial vertex of the edge.  

To define the Morse matching on $\ud{n}{\Gamma}$, we will need the following definitions.  Let $c$ be a cell in $\ud{n}{\Gamma}$.  A vertex $v$ in $c$ is \emph{blocked} in $c$ if either $v = \ast$ or $e(v) \cap c' \neq \emptyset$ for some other vertex or edge $c'$ of $c$ aside from $v$; we say $v$ is \emph{blocked by} $\ast$ or $c'$ in $c$, respectively.  If $v$ is not blocked in $c$, it is called \emph{unblocked} in $c$.  Equivalently, $v$ is unblocked in $c$ if and only if $c \cup \{e\} - \{v\}$ is also a cell in $\ud{n}{\Gamma}$.  

An edge $e$ in $c$ is called \emph{non-order-respecting} in $c$ if:
\begin{enumerate}
\item $e \not \subseteq T$, or

\item there is some vertex $v$ of $c$ such that $v \in ( \tau e, \iota e)$
and $e(v) \cap e = \tau e$. 
\end{enumerate}
An edge $e$ in $c$ is \emph{order-respecting} in $c$ otherwise.  

\subsection{The Morse matching}
\label{sec:Morsematching}

In this subsection, we will define a Morse matching on $\ud{n}{\Gamma}$ for an arbitrary positive integer $n$ and finite connected graph $\Gamma$.  

Suppose that we are given a cell $c = \{ c_1, \ldots, c_n \}$ in $\ud{n}{\Gamma}$.  Assign each cell $c_i$ in $c$ a number as follows.  A vertex of $c$ is given the number from the above traversal of $T$.  An edge $e$ of $c$ is given the number for $\iota e$.  Arrange the cells of $c$ in a sequence $\mathcal{S}$, from the least- to the greatest-numbered.  The following definition of a Morse matching $W$ is equivalent to the definition of $W$ from \cite{FarleySabalka1}, by Theorem 3.6 of the same paper.

\begin{definition} \label{def:critical}
We define a Morse matching $W$ on $\ud{n}{\Gamma}$ as follows:
\begin{enumerate}
\item If an unblocked vertex occurs in $\mathcal{S}$ before all of the order-respecting edges in $c$ (if any), then $W(c)$ is obtained from $c$ by replacing the minimal unblocked vertex $v \in c$ with $e(v)$.  In particular, $c$ is redundant.

\item If an order-respecting edge occurs in $\mathcal{S}$ before all of the unblocked vertices of $c$ (if any), then $c \in \operatorname{image} (W)$, i.e., $c$ is collapsible.  The cell $W^{-1}(c)$ is obtained from $c$ by replacing the minimal order-respecting edge $e$ with $\iota e$.

\item If there are neither unblocked vertices nor order-respecting edges
in $c$, then $c$ is critical.
\end{enumerate}
\end{definition}

\begin{example}
Figure \ref{fig:classic} depicts three different cells of $\ud{4}{T_{min}}$ for the given tree $T_{min}$.  In each case, the number label of each vertex and of the initial vertex of each edge are shown, in the sense mentioned above.

\begin{figure}[!h]
\centering

\subfigure[]{
  \includegraphics{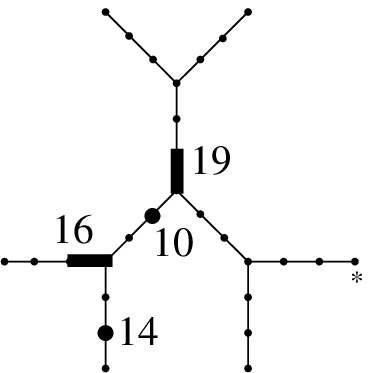}
}\hfill
\subfigure[]{
  \includegraphics{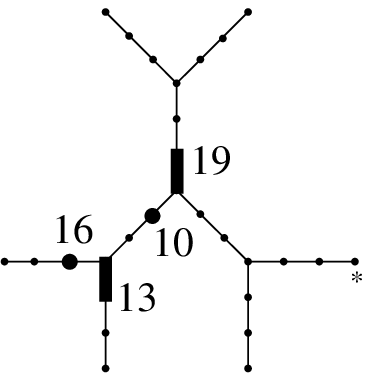}
}\hfill
\subfigure[]{
  \includegraphics{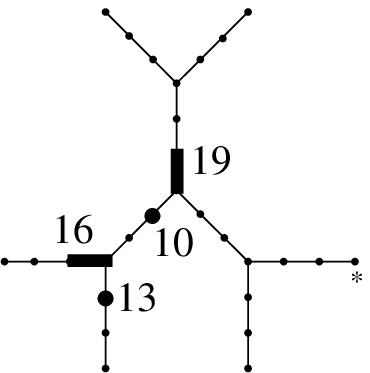}
}

\caption{Three different cells of $\ud{n}{T}$.  The first is redundant, the second collapsible, and the third critical.}
\label{fig:classic}
\end{figure}

We let $c_{1}$ denote the cell depicted in (a). The vertex numbered $10$ is blocked.  The vertex numbered $14$ is unblocked, so $c_{1}$ is redundant.  Note that edge $e(16)$ is order-respecting and edge $e(19)$ is non-order-respecting.  We get $W(c_1)$ by replacing vertex $14$ with the unique edge in $T$ having vertex $14$ as its initial vertex, i.e., $[13,14]$.  

Let $c_2$ denote the cell depicted in (b).  The vertex numbered $10$ in $c_2$ is blocked.  The edge $e(13)$ is order-respecting, so $c_2$ is collapsible.  Note that vertex $16$ is blocked and edge $e(19)$ is non-order-respecting.  The description of $W^{-1}$ above implies that $W^{-1}(c_2)$ is obtained from $c_2$ by replacing edge $e(13)$ with its initial vertex.

The cell depicted in (c) is critical since vertices $10$ and $13$ are both blocked, and the edges $e(16)$ and $e(19)$ are non-order-respecting.
\end{example}

\section{Presentations for Graph Braid Groups}\label{sec:main}

\subsection{Preliminary notation}\label{subsec:prelim}

In this subsection, we include some final notation before we proceed to the main argument.

We will sometimes use set notation in order to describe edge-paths.  For instance, suppose that $\Gamma$ is the
tree from Figure \ref{fig:classic}.  The edge-path $\{ [\ast, 9], 5, 6, 10 \} \subseteq \ud{n}{\Gamma}$ 
starts at the vertex $\{ 5, 6, 9, 10 \}$ and ends at $\{ \ast, 5, 6, 10 \}$.  The edges of this edge-path are determined
by moving the vertex whose original position is $9$ back to the basepoint $\ast$.  (There are $6$ edges in all.)  
All of the other vertices are fixed.

Let $c$ be a $k$-cell in $\ud{n}{\Gamma}$.  If $c$ has no unblocked vertices, define $r(c) = c$.  Otherwise, if $v$ is the smallest unblocked vertex in $c$, define $r(c)$ to be the $k$-cell that may be obtained from $c$ by replacing $v$ with $\tau(e(v))$ (i.e., the vertex adjacent to $v$ along the edge-path $[\ast, v] \subseteq T$, where $T$ is the maximal tree).  For instance, if $c = \{ \ast, 1, [7,8], 19 \}$ where $c \subseteq \ud{4}{T_{min}}$ (as in Figure \ref{fig:classic}), then $r(c) = \{ \ast, 1, [7,8], 9 \}$. We let $r^{\infty}(c)$ denote the result of repeatedly applying $r$ to $c$ until the output stabilizes (as it obviously must).

We will frequently need to describe cells in $\ud{n}{\Gamma}$ using vector notation, as was done in \cite{FarleySabalka1}.  Let $X$ be an essential vertex in $T$, the maximal subtree of $\Gamma$.  We let $\vec{a}$ be a vector with $deg_{T}(X)-1$ entries, all of which are non-negative integers.  Let $1 \leq i \leq deg_{T}(X) - 1$, and let $(\vec{a})_{j}$ denote the $j$th coordinate of $\vec{a}$, for $1 \leq j \leq deg_{T}(X)-1$.   We write $|\vec{a}|$ do denote the sum $\sum_{i = 1}^{deg_T(X)-1} (\vec{a})_i$ of entries of $\vec{a}$.  By $X_{i}[\vec{a}]$, we denote a configuration $c$ (or subconfiguration) of vertices and edges such that:

\begin{enumerate}
\item There is an edge $e$ such that $\tau(e) = X$, and $e$ points in direction $i$ from $X$.  There are $(\vec{a})_{i} - 1$ vertices blocked by $\iota(e)$;

\item For each $j \neq i$, there are $(\vec{a})_{j}$ vertices lying in direction $j$ from $X$, and blocked by the edge $e$.
\end{enumerate}
Note that if $c \in \ud{n}{\Gamma}$, then $1 \leq |\vec{a}| \leq n$.  When we use this vector notation, we typically assign each essential vertex in $\Gamma$ a capital letter of the alphabet, in alphabetical order.  

For instance, suppose that $T = \Gamma$ is as in Figure \ref{fig:classic}.  We write $A_{2}[2,2]$ in place of $\{ 4, 5, [3,7], 8 \}$, $B_{2}[1,3]$ in place of $\{ 10, [9,19], 20, 21 \}$, and so on.  We can also use additive notation to describe higher-dimensional cells.  For instance, the critical $2$-cell depicted in Figure \ref{fig:classic}(c) can be described by the notation $B_{2}[1,1] + C_{2}[1,1]$. The collapsible $2$-cell in Figure \ref{fig:classic}(b) can be described as $B_{2}[1,1] + C_{1}[1,1]$.

We will also occasionally extend this additive notation to describe cells containing deleted edges $e$ and vertices blocked at the basepoint. Thus, for instance, we might write $X_{i}[\vec{a}] + e_{1} + e_{2} + k\ast$ to denote a cell consisting of the configuration $X_{i}[\vec{a}]$ along with the deleted edges $e_{1}$ and $e_{2}$, and $k$ vertices $\ast, 1, \ldots, k-1$ blocked at the basepoint.  In practice, we will typically leave off $k \ast$ when describing a cell in $\ud{n}{\Gamma}$, since this should cause no confusion.

It is easy to deduce under fairly general assumptions (for instance, see Assumption \ref{assumption}) that a formal sum consisting of terms of the form $X_{i}[\vec{a}]$ (where $X$ is an essential vertex), and others of the form $e$ (where $e$ is a deleted edge), is critical if and only if each term of the form $X_{i}[\vec{a}]$ satisfies $(\vec{a})_{j} > 0$, for some $j<i$.  (Note:  we also assume of course that $\vec{a}$ is a vector with non-negative integer entries, and that $(\vec{a})_{i} > 0$.)  Indeed, the critical cells of $\ud{n}{\Gamma}$ are in one-to-one correspondence with such formal sums.  (A more explicit version of this statement appears as Proposition 3.9 of \cite{FarleySabalka1}.)      

We need some final definitions for technical reasons.  We multiply vectors by constants in the obvious way.  We also define subtraction of a constant from a vector, as follows.  Let $\vec{a} - 1$ be the result of subtracting $1$ from the first non-zero entry of $\vec{a}$.  We let $\vec{a} - 2= (\vec{a} - 1) - 1$, and so on.  We note that this notation must be used with caution, since $(\vec{a} + \vec{b}) - 1 \neq (\vec{a} - 1) + \vec{b}$ in general.  Finally, we let $\delta_{i}$ denote the vector such that $(\delta_{i})_{i}=1$ and $(\delta_{i})_{j} = 0$, for all $j \neq i$.  The number of components of $\delta_{i}$ will always be clear from the context.

\subsection{Relative flow, cost, and the main theorem} 
\label{subsec:flowcostmain}

We now give some examples, which are intended to the smooth the way for our proofs of
Theorem \ref{thm:biggie}, Proposition \ref{prop:vanishing}, and Proposition \ref{prop:cost}.
The following example is intended, above all, to motivate Definition \ref{def:rflow}.  

\begin{figure} [!h]
\begin{center}
\includegraphics{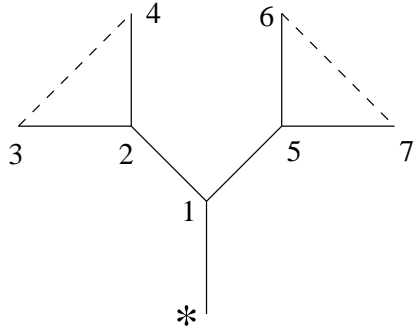}
\end{center}
\caption{A graph $\Gamma$ with a maximal tree $T$ specified.}
\label{fig:balloongraph1}
\end{figure}

\begin{example} \label{ex:first}
We let $\Gamma$ be the graph depicted in Figure \ref{fig:balloongraph1}, and consider the two-strand braid group of $\Gamma$.  Let $T$ denote the maximal tree in $\Gamma$ consisting of all edges and vertices of $\Gamma$,except for the dashed edges in Figure \ref{fig:balloongraph1}.  The figure also indicates a choice of embedding of $T$ in $\mathbb{R}^{2}$, a choice of basepoint in $T$, and the resulting numbering of the vertices of $\Gamma$.

It is a simple matter to list the critical $1$-cells and $2$-cells.  The critical $1$-cells are as follows:
	$$ \{ 3, [2,4] \}, \{ 2, [1,5] \}, \{ 6, [5,7] \}, \{ \ast, [3,4] \}, \{ \ast, [6,7] \}.$$
There is a single critical $2$-cell:
	$$ \{ [3,4], [6,7] \}.$$
It follows from Theorem \ref{thm:group} that $B_{2} \Gamma$ has a presentation with five generators and one relator.

\begin{figure} [!t]
\begin{center}
\includegraphics{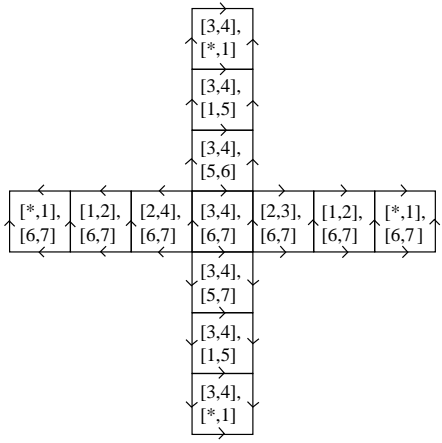}
\end{center}
\caption{The picture represents an intermediate stage in the computation of the relator in a presentation for $B_{2}\Gamma$.  The central square $c$ is a critical $2$-cell, and each of the arms is a sequence of squares, across which the faces of $c$ are pushed by the Morse matching.}
\label{fig:plus}
\end{figure}

Now we will find the relator in this presentation.  The critical $2$-cell $\{ [3,4], [6,7] \}$ can be identified with the product $[3,4] \times [6,7]$.  If we recall that $4$ is the initial vertex of $[3,4]$ and $7$ is the initial vertex of $[6,7]$ (and we treat the factor of $[3,4]$ as the first factor), then we arrive at the central oriented square in Figure \ref{fig:plus}.  The bottom left vertex of this square is $\{ 4, 7 \}$, the upper left corner is $\{ 4, 6 \}$, the upper right corner is $\{ 3, 6 \}$, and the lower right corner is $\{ 3, 7 \}$.

Consider the bottom horizontal face of the critical $2$-cell $\{ [3,4], [6,7] \}$.  This is the $1$-cell $\{ [3,4], 7 \}$.  Clearly this $1$-cell is redundant; its image under the Morse matching is $\{ [3,4], [5,7] \}$ (as pictured in Figure \ref{fig:plus}; this cell is immediately below $\{ [3,4], [6,7] \}$).  The Morse matching therefore tells us to push the $1$-cell $\{[3,4], 7 \}$ across the $2$-cell $\{ [3,4], [5,7] \}$.  In the process, $\{ [3,4], 7 \}$ is smeared across the edge-path
	$$\{ 4, [5,7] \} \{ [3,4], 5 \} \{ 3, [5,7] \}^{-1},$$
which makes up the sides and bottom of the $2$-cell $\{ [3,4], [5,7] \}$.  We repeat this procedure two more times, pushing $\{ [3,4], 5 \}$ across $\{ [3,4], [1,5] \}$, and finally pushing $\{ [3,4], 1 \}$ across $\{ [3,4], [\ast, 1] \}$.  The cumulative effect is to smear the original $1$-cell $\{ [3,4], 7 \}$ across the edge-path
	\begin{align}
	&\{ 4, [\ast, 7] \} \{ [3,4], \ast \} \{ 3, [\ast, 7] \}^{-1}.  \tag{P1} \label{1}
	\end{align}
Note that we haven't yet completed the process of pushing redundant $1$-cells across their matching collapsible $2$-cells, since $\{ 4, [5,7] \}$, $\{ 4, [1,5] \}$, $\{ 3, [1,5] \}$, and $\{ 3, [5,7] \}$ are all redundant $1$-cells.  We will finish after introducing a lemma (Lemma \ref{lemma:redundant} below).  
For now, we note that Figure \ref{fig:plus} shows that the Morse matching smears the left, top, and right sides of the critical $2$-cell $\{ [3,4], [6,7] \}$ across the edge-paths
	\begin{align}
	&\{ [\ast,4], 7 \} \{ \ast, [6,7] \} \{ [\ast, 4], 6 \}^{-1}, \tag{P2} \label{2}\\
	&\{ 4, [\ast,6] \} \{ [3,4], \ast \} \{ 3, [\ast, 6] \}^{-1}, \tag{P3} \label{3}\\
	&\{ [\ast,3], 7 \} \{ \ast, [6,7] \} \{ [\ast, 3], 6 \}^{-1}, \tag{P4} \label{4}
	\end{align}
respectively.

\end{example}

\begin{lemma} \label{lemma:redundant}
(Redundant $1$-cells lemma)  Assume $c = \{ v_{1}, \ldots, v_{n-1}, e \}$ is a redundant $1$-cell in $\ud{n}{\Gamma}$.  We assume that the vertex $v_{1}$ is the smallest unblocked vertex in $c$.  Assume that $T \subseteq \Gamma$ is the maximal tree in $\Gamma$.
\begin{enumerate}
\item If $(\tau(e(v_{1})), v_{1}) \cap \{ v_{2}, \ldots, v_{n-1}, \tau(e), \iota(e) \} = \emptyset$, then $c \dot{\rightarrow} r(c)$.

\item If $\{ v_{1}, \ldots, v_{n-1} \} \cap (\tau(e), \iota(e)) = \emptyset$ and $e \subseteq T$, then $M^{\infty}(c) = 1$.
\end{enumerate}
\end{lemma}

\begin{proof}
Statement (1) is called ``the redundant 1-cells lemma" in \cite{FarleySabalka1}; we refer the reader to that source for a proof.  
Statement (2) is proved in \cite{FarleySabalka1} as an intermediate step in proving (1).
\end{proof}

\begin{example}
We apply the previous lemma to the edge paths from the end of Example \ref{ex:first}.  Edge-path \eqref{1} contains the subpath $\{4, [\ast, 7]\}$, which is made up of the edges $\{ 4, [5,7] \}$, $\{4, [1,5] \}$, and $\{ 4, [\ast, 1]\}$.  The first and last of these edges flow to the trivial word by Lemma \ref{lemma:redundant}(2).  The remaining edge, $\{4, [1,5] \}$, flows to $\{2, [1,5]\}$ by Lemma \ref{lemma:redundant}(1).  The edge $\{ [3,4], \ast \}$ is critical, and therefore stable under the flow.  We can apply the flow to the subpath $\{3, [\ast, 7] \}^{-1}$ just as we did to the subpath $\{ 4, [\ast, 7] \}$.  The result is that $\{3, [\ast, 7]\}^{-1}$ flows to $\{2, [1,5] \}^{-1}$.  We conclude that \eqref{1} flows to      
	$$\{ 2, [1,5] \} \{ [3,4], \ast \} \{ 2, [1,5] \}^{-1}.$$
Similarly, the edge-paths \eqref{2}, \eqref{3}, and \eqref{4} flow, respectively, to:
	$$\{ \ast, [6,7] \}, \qquad \{ 2, [1,5] \} \{ [3,4], \ast \} \{ 2, [1,5] \}^{-1}, \qquad \{ \ast, [6,7] \}.$$
If we rename the critical $1$-cells $\{ 2, [1,5] \}$, $\{ 3, [2,4] \}$, $\{ 6, [5,7] \}$, $\{ [3,4], \ast \}$, and $\{ \ast, [6,7] \}$ as $A$, $B$, $C$, $\mathbf{e}_{1}$, and $\mathbf{e}_{2}$ (respectively), then
	$$ B_{2}\Gamma \cong \langle A, B, C, \mathbf{e}_{1}, \mathbf{e}_{2} \mid 
	[ \mathbf{e}_{2}, A\mathbf{e}_{1}A^{-1} ] \rangle,$$
by Theorem \ref{thm:group}.  It is easy to rewrite this presentation to show that $B_{2} \Gamma \cong F_{3} \ast \mathbb{Z}^{2}$, where $F_{3}$ is the free group on three generators.
\end{example} 

\begin{definition} \label{def:rflow}
(Relative Flow)  Let $\mathcal{V} = \{ v_{1}, \ldots, v_{n-1} \}$ be a collection of vertices, and  let $e$ be a closed edge in $\Gamma$.  We assume that $e$ is non-order-respecting in $\{ v_{1}, \ldots, v_{n-1}, e \}$.  We define an edge-path $f( \mathcal{V}; \iota(e))$ inductively as follows.  If all vertices in $\{ v_{1}, v_{2}, \ldots, v_{n-1}, e \}$ are blocked, then $f(\mathcal{V}; \iota(e))$ is the trivial path at $\{ v_{1}, \ldots, v_{n-1}, \iota(e) \}$.  Otherwise, let $v$ be the smallest unblocked vertex in the configuration $\{ v_{1}, \ldots, v_{n-1}, e \}$.  Without loss of generality, $v = v_{1}$.  We set $\mathcal{V}' = \{ \tau(e(v_{1})), v_{2}, \ldots, v_{n-1} \}$.  By definition, 
	$$ f(\mathcal{V}; \iota(e)) = \{ e(v_{1}), v_{2}, \ldots, v_{n-1}, \iota(e) \} f( \mathcal{V}'; \iota(e)).$$ 
The latter equation completely determines $f( \mathcal{V}; \iota(e))$.  We can define $f(\mathcal{V}; \tau(e))$ analogously: simply replace $\iota(e)$ with $\tau(e)$ in the above notation (while leaving the definition of $\mathcal{V}'$ unchanged).

The edge-path $f(\mathcal{V}; \iota(e))$ is called the \emph{flow of $\mathcal{V}$ relative to $\iota(e)$}.  We write $c(\mathcal{V}; \iota(e)) = M^{\infty}(f(\mathcal{V};\iota(e)))$.  The string $c(\mathcal{V}; \iota(e))$ is a word in the free group on critical $1$-cells called the \emph{cost of $\mathcal{V}$ relative to $\iota(e)$}.  
\end{definition}

\begin{note} \label{note:intuitive}
If $\mathcal{V} = \{ v_{1}, \ldots, v_{n-1} \}$ is a collection of vertices and $e$ is a non-order-respecting edge in the configuration $\{ v_{1}, \ldots, v_{n-1}, e \}$, then the relative flow is an edge-path in $\ud{n}{\Gamma}$ with a simple direct description.  We start with the vertex $\{ v_{1}, \ldots, v_{n-1}, \iota(e) \}$ in $\ud{n}{\Gamma}$ and repeatedly move the vertices of $\{ v_{1}, \ldots, v_{n-1}, \iota(e) \}$ while following these rules:
\begin{enumerate}
\item The vertex $\iota(e)$ never moves.

\item At each stage, we move the smallest unblocked vertex $v$ one step within the tree $T$ towards the basepoint $\ast$.  Here a vertex $v$ is considered blocked (for the sake of this definition) if the vertex is blocked in the ordinary sense by $e$.
\end{enumerate}
\end{note}

\begin{note}
Although the notation $f(\mathcal{V}; \iota(e))$ features a vertex ($\iota(e)$) in the second place, we need to keep track of both an edge ($e$) and an endpoint ($\iota(e)$, in this case).  We write $\iota(e)$ for the sake of compression.  It would be ambiguous to specify only a vertex, since the edge $e$ is allowed to block the movements of vertices.  Nevertheless, we will sometimes write $c(\mathcal{V}; v_{n})$ (i.e., specify only a vertex in the second coordinate) if the edge $e$ in the definition of the relative flow never blocks the movements of any vertices, or if the identity of the edge $e$ is clear from the context.
\end{note}

\begin{example}\label{example:letterh}
Consider the tree $T$ in Figure \ref{fig:letterh}.  Let $n=4$.  The given $T$ is sufficiently subdivided for $n$.  We let $\mathcal{V} = \{ 3,4,7 \}$ and $e = [6,10]$, and consider the relative flow $f( \mathcal{V}; \iota(e))$.  We have
	$$ f(\mathcal{V}; \iota(e)) = \{ [\ast,3], 4, 7, 10\} \{ \ast, [1,4], 7, 10\}. $$
Thus, $f(\mathcal{V}; \iota(e))$ is an edge-path of length $6$.

The cost $c(\mathcal{V}; \iota(e))$ is obtained by applying the flow to $f( \mathcal{V}; \iota(e))$.  It is clear in this case that $c( \mathcal{V}; \iota(e))=1$, since Lemma \ref{lemma:redundant}(2) applies to each of the six $1$-cells in $f(\mathcal{V}; \iota(e))$.

Now consider the flow $f(\mathcal{V}; \iota(e))$, where $\mathcal{V} = \{ 4, 7, 13 \}$ and $e = [6,10]$.  We get 
	$$f(\mathcal{V}; \iota(e)) = \{ [\ast,4], 7, 10, 13 \} \{ \ast, [1,13], 7, 10\},$$
an edge-path of length $7$. We now compute the cost $c(\mathcal{V}; \iota(e))$.  Each of the edges in $f(\mathcal{V}; \iota(e))$ flows to $1$, by Lemma \ref{lemma:redundant}(2), with the exception of $\{ \ast, [3,13], 7, 10 \}$, which flows to $\{ \ast, [3,13], 4, 5 \}$, by repeated applications of Lemma \ref{lemma:redundant}(1).  Therefore, $c(\mathcal{V}; \iota(e)) = \{ \ast, [3,13], 4, 5 \}$.
\end{example}

\begin{figure} [!h]
\begin{center}
\includegraphics{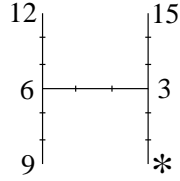}
\end{center}
\caption{Example \ref{example:letterh} is a computation of the relative flow $f(\mathcal{V}; \iota(e))$ for certain choices of vertices $\mathcal{V}$ and edges
$e$.}
\label{fig:letterh}
\end{figure}

\begin{figure}[!h]
\begin{center}
\hspace{-1.1in}
\input{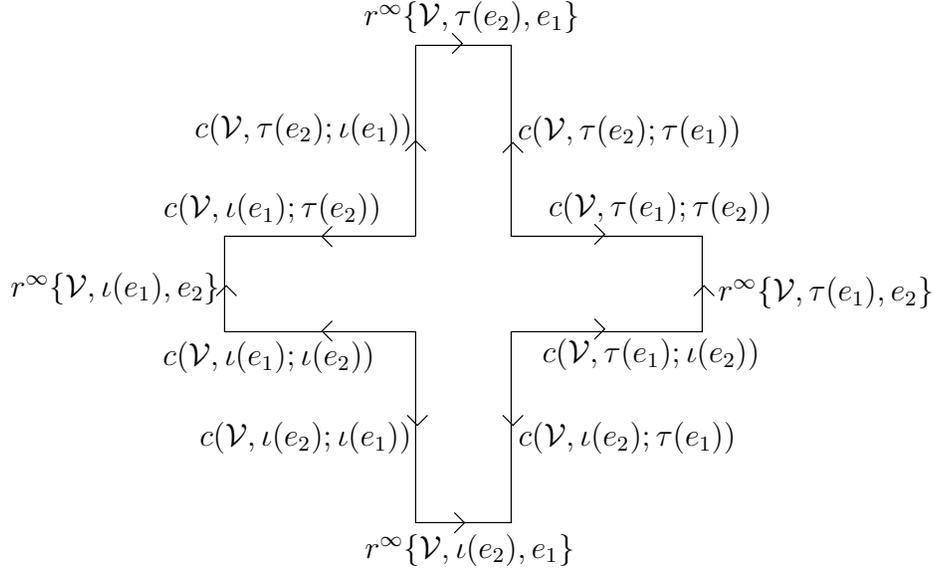}
\end{center}
\caption{This figure shows the general form of any relator in a graph braid group.}
\label{fig:biggie}
\end{figure}

\begin{theorem} \label{thm:biggie} (General Presentation Theorem)
Let $\Gamma$ be a finite connected graph.  Let $n \in \mathbb{N}$.  Assume that $\Gamma$ is sufficiently subdivided for $n$.  The braid group $B_{n}\Gamma$ has a presentation $\mathcal{P} = \langle \Sigma \mid \mathcal{R} \rangle$ where $\Sigma$ is the set of critical $1$-cells in $\ud{n}{\Gamma}$, and $\mathcal{R}$ is in one-to-one correspondence with the set of critical $2$-cells in $\ud{n}{\Gamma}$.  A critical $2$-cell $\{ v_{1}, \ldots, v_{n-2}, e_{1}, e_{2} \}$ corresponds to the relator in Figure \ref{fig:biggie}, where $\mathcal{V} = \{ v_{1}, \ldots, v_{n-2} \}$.
\end{theorem}

\begin{proof}
We need only show that the critical $2$-cell $\{ v_{1}, \ldots, v_{n-2}, e_{1}, e_{2} \}$ corresponds to the relator appearing in Figure \ref{fig:biggie}; the rest of the statement is a corollary of Theorem \ref{thm:group}.

Let us consider the $1$-cell $\{ v_{1}, \ldots, v_{n-2}, \iota(e_{1}), e_{2} \}$.  We wish to show that
	\begin{align}
	M^{\infty} \{ \mathcal{V}, \iota(e_{1}), e_{2} \} 
	&= c(\mathcal{V}, \iota(e_{1}); \iota(e_{2})) r^{\infty} \{\mathcal{V}, \iota(e_{1}), e_{2} \}
	c(\mathcal{V}, \iota(e_{1}), \tau(e_{2}))^{-1}. \tag{*} \label{ast}
	\end{align}
(The right side of the latter equation labels the left arm of the plus sign in Figure \ref{fig:biggie}.)
We begin by noting that $\{\mathcal{V}, \iota(e_{1}), e_{2} \}$ cannot be collapsible.  Indeed, since $\{\mathcal{V}, e_{1}, e_{2} \}$ is critical, it must be that $e_{2}$ is non-order-respecting, so either i)  $e_{2}$ is not contained in the maximal tree $T$, or ii) $e_{2} \subseteq T$, and there is a vertex $v$ such that $\tau(e_{2}) < v < \iota(e_{2})$, and $e(v) \cap e_{2} = \tau(e_{2})$.  Moreover, both of these conditions are inherited by $\{ v_{1}, \ldots, v_{n-2}, \iota(e_{1}), e_{2} \}$.

It follows that $\{\mathcal{V}, \iota(e_{1}), e_{2} \}$ is either critical or redundant.  We assume first that $\{\mathcal{V}, \iota(e_{1}), e_{2} \}$ is critical.  In this case, each vertex in $\{\mathcal{V}, \iota(e_{1}), e_{2} \}$ is blocked, so, by definition, $f(\mathcal{V}, \iota(e_{1}); \iota(e_{2}))$ and $f(\mathcal{V}, \iota(e_{1}); \tau(e_{2}))$ are both trivial paths, $r^{\infty}\{\mathcal{V}, \iota(e_{1}), e_{2} \} = \{\mathcal{V}, \iota(e_{1}), e_{2} \}$, and $M^{\infty} \{\mathcal{V}, \iota(e_{1}), e_{2} \} = \{\mathcal{V}, \iota(e_{1}), e_{2} \}$.  It follows that \eqref{ast} is satisfied in the current case.

Now suppose that $\{\mathcal{V}, \iota(e_{1}), e_{2} \}$ is redundant, and so has an unblocked vertex.  We write $\iota(e_{1}) = v_{n-1}$, and let $v_{i}$ denote the smallest of all unblocked vertices in $\{ v_{1}, \ldots, v_{n-1}, e_{2} \}$.  (We note that $v_{i}$ could be $v_{n-1}$, i.e., $\iota(e_{1})$.)   
We let 
	\begin{align*}
	\widetilde{\mathcal{V}} &= \{ v_{1}, \ldots, v_{n-1} \}, \\
	e(\widetilde{\mathcal{V}}) &= \{ v_{1}, \ldots, e(v_{i}), \ldots, v_{n-1} \}, \\
	\tau(e(\widetilde{\mathcal{V}})) &= \{ v_{1}, \ldots, \tau(e(v_{i})), \ldots, v_{n-1} \}.
	\end{align*}
By definition 
	$$ \{ \widetilde{\mathcal{V}}, e_{2} \} 
	\rightarrow \{ e(\widetilde{\mathcal{V}}), \iota(e_{2}) \} r \{ \widetilde{\mathcal{V}}, e_{2} \}
	\{ e(\widetilde{\mathcal{V}}), \tau(e_{2}) \}^{-1}.$$
Now we apply $M^{\infty}$ to both sides.  By induction, we have that $M^{\infty} \{ \widetilde{\mathcal{V}}, e_{2} \}$ is equal to 
	$$M^{\infty} \{ e(\widetilde{\mathcal{V}}), \iota(e_{2}) \} c(\tau(e(\widetilde{\mathcal{V}})); 
	\iota(e_{2}) ) r^{\infty} \{ \widetilde{\mathcal{V}}, e_{2} \}
	c( \tau(e(\widetilde{\mathcal{V}})); \tau(e_{2}) )^{-1}
	M^{\infty}\{ e(\widetilde{\mathcal{V}}), \tau(e_{2}) \}^{-1}.$$
Finally, we note that
	\begin{align*}
	M^{\infty}\{ e(\widetilde{\mathcal{V}}), \iota(e_{2}) \} 
	c( \tau(e(\widetilde{\mathcal{V}})); \iota(e_{2}))  &=  
	c(\widetilde{V}; \iota(e_{2})) \\   
	M^{\infty}\{ e(\widetilde{\mathcal{V}}), \tau(e_{2}) \} 
	c( \tau(e(\widetilde{\mathcal{V}})); \tau(e_{2}))  &=  
	c( \widetilde{V}; \tau(e_{2}))      
	\end{align*}
Condition \eqref{ast} follows directly.  One applies the same reasoning to each of the three remaining faces of $\{ v_{1}, \ldots, v_{n-2}, e_{1}, e_{2} \}$.
\end{proof}

\begin{example}
We let $\Gamma$ be the tree from Figure \ref{fig:classic}, and let $n=4$.  It is possible to show that the space $\ud{4}{\Gamma}$ has $24$ critical $1$-cells and $6$ critical $2$-cells.  We will compute the relator corresponding to the critical $2$-cell $\{ 13, [12, 16], 10, [9,19] \}$.

We compute each of the eight costs.
\begin{enumerate}

\item We compute $f(10,13,16;9)$:
	$$f(10,13,16;9) = \{ 9, 10, [11,13], 16 \}\{ 9, 10, 11, [12,16] \}.$$
Lemma \ref{lemma:redundant}(2) implies that all of the edges in this edge-path flow to $1$, so $c(10,13,16;9)=1$.

\item The flow $f(10,13,16;19)$ is equal to $f(10,13,16;9)$ (but with $19$ in place of $9$), 
so $c(10,13,16;19) = c(10,13,16;9) =1$.  The costs $c(10,12,13;9)$, $c(10,12,13;19)$, $c(9,10,13;16)$, and $c(9,10,13;12)$ are trivial also, since the corresponding flows are again made up of edges satisfying the hypotheses of Lemma \ref{lemma:redundant}(2).

\item We compute the relative flow $f(10,13,19;16)$:
	$$ f(10,13,19;16) = \{ [\ast, 10], 13, 16, 19 \} \{ \ast, [1,19], 13, 16 \}.$$
All of the edges in this edge-path satisfy the hypotheses of Lemma \ref{lemma:redundant}(2), with the exception of $\{ \ast, [9,19], 13, 16 \}$. If we apply Lemma \ref{lemma:redundant}(1) to the latter edge repeatedly, we eventually conclude that
	$$ M^{\infty}\{ \ast, [9,19], 13, 16 \} = \{ \ast, [9,19], 10, 11 \}.$$
Since all of the other edges in $f(10,13,19;16)$ flow to $1$, we have $c(10,13,19;16) = \{ \ast [9,19], 10, 11 \}$.  Exactly the same argument shows that $c(10,13,19;12) = \{ \ast, [9,19], 10, 11 \}$.
\end{enumerate}

It is easy to see that
	\begin{align*}
	r^{\infty}\{ 13, 16, 10, [9,19] \} &= \{ 10, 11, 12, [9,19] \}; \\ 
	r^{\infty} \{ 13, 12, 10, [9,19] \} &= \{ 10, 11, 12, [9,19] \};  \\              
	r^{\infty}\{ 13, [12,16],  10, 9 \} &= \{ \ast, 1, 13, [12,16] \}; \\
	r^{\infty} \{ 10, 13, [12,16], 19 \}  &= \{ \ast, 1, 13, [12,16] \}.
	\end{align*}
We therefore arrive at the following commutator relation, by Theorem \ref{thm:biggie}:
	$$ \left[ \{ \ast, 1, 13, [12,16] \}, \{ 10, 11, 12, [9,19] \}^{-1} \{ \ast, [9,19], 10, 11 \} \right].$$
\end{example}

\subsection{Calculation of costs}

In the previous example, at least one of the costs at any corner of a given critical $2$-cell was trivial.  This is true under fairly general hypotheses.

\begin{assumption} \label{assumption}
Let $n \in \mathbb{N}$.  Assume that $\Gamma$ is subdivided and the maximal tree $T \subseteq \Gamma$ is chosen in such a way that:
	\begin{enumerate}
	\item Each vertex of degree $1$ in $T$ has degree at most $2$ in $\Gamma$, and
	\item $T$ is sufficiently subdivided for $n$.
	\end{enumerate}
\end{assumption}

\begin{proposition} \label{prop:vanishing}(Vanishing Costs) 
Assume that $\Gamma$ and $T$ satisfy Assumption \ref{assumption} for $n$.  Let $c = \{ v_{1}, \ldots, v_{n-2}, e_{1}, e_{2} \}$ be a critical $2$-cell in $\ud{n}{\Gamma}$.  Let $\alpha, \beta \in \{ \iota, \tau \}$.  One of the costs $c(v_{1}, \ldots, v_{n-2}, \alpha(e_{1}); \beta(e_{2}))$, $c(v_{1}, \ldots, v_{n-2}, \beta(e_{2}); \alpha(e_{1}))$ is trivial.  More specifically, $c( v_{1}, \ldots, v_{n-2}, \alpha(e_{1}); \beta(e_{2})) = 1$ in all of the following cases.
\begin{enumerate}
\item $e_{1}$, $e_{2} \subseteq T$ and
	\begin{enumerate}
	\item $\tau(e_{2}) < \tau(e_{1}) < \iota(e_{1}) < \iota(e_{2})$, or
	\item $\tau(e_{1}) < \iota(e_{1}) < \tau(e_{2}) < \iota(e_{2})$.
	\end{enumerate}
\item $e_{1}$, $e_{2} \not \subseteq T$ and $\alpha(e_{1}) < \beta(e_{2})$.
\item $e_{1} \not \subseteq T$, $e_{2} \subseteq T$, and
	\begin{enumerate}
	\item $\alpha(e_{1}) \wedge \tau(e_{2}) = \tau(e_{2})$, or
	\item $\alpha(e_{1}) < \tau(e_{2})$.
	\end{enumerate}
\item $e_{1} \subseteq T$,
$e_{2} \not \subseteq T$, and there is some vertex $v \in T$ such that $0 < d(v, \tau(e_{1})) < d(v, \beta(e_{2}))$.
\end{enumerate}
\end{proposition}

\begin{proof}
We first assume that $e_{1}$, $e_{2} \subseteq T$.  In this case, since all vertices in $c$ are blocked, each vertex is either blocked by $e_{1}$, blocked by $e_{2}$, or blocked at the basepoint.  (Here we mean that a vertex $v$ is blocked by $e_{1}$ if $v$ is either blocked by $e_{1}$ in the usual sense, or that $v$ is blocked by a vertex that is blocked by $e_{1}$, etc.)  

Assume that $\tau(e_{2}) < \tau(e_{1}) < \iota (e_{1}) < \iota (e_{2})$.  We note first that the vertex $\alpha(e_{1})$ and each vertex blocked by $e_{1}$ will eventually be blocked by $e_{2}$ after it has moved (possibly several times) by our assumption.  It follows that the relative flow $f(v_{1}, \ldots, v_{n-2}, \alpha(e_{1}); \beta(e_{2}))$ is the edge-path determined by moving each of the vertices blocked at $e_{1}$ and $\alpha(e_{1})$ in order, until all are blocked at $e_{2}$. Each edge $\{ v'_{1}, \ldots, v'_{n-2}, \beta(e_2), e \}$ in this edge-path has the property that $\{ v'_{1}, \ldots, v'_{n-2}, \beta(e_2) \} \cap ( \tau(e), \iota(e) ) = \emptyset$.  It follows that $M^{\infty} \{ v'_{1}, \ldots, v'_{n-2}, \beta(e_2), e \} = 1$, and so $c(v_{1}, \ldots, v_{n-2}, \alpha(e_{1}); \beta(e_{2})) =1$, as claimed.

Assume now that $\tau(e_{1}) < \iota(e_{1}) < \tau(e_{2}) < \iota(e_{2})$.  In this case, the vertex $\alpha(e_{1})$ and each of the vertices blocked by $e_{1}$ will eventually be blocked at the basepoint after moving several times.  The relative flow $f(v_{1}, \ldots, v_{n-2}, \alpha(e_{1}); \beta(e_{2}))$ is the edge-path determined by moving $\alpha(e_{1})$ and each of the vertices blocked by $e_{1}$ in order until all are blocked at the basepoint.  Each edge $\{ v'_{1}, \ldots, v'_{n-2}, \beta(e_2), e \}$ in this edge-path has the property that $\{ v'_{1}, \ldots, v'_{n-2}, \beta(e_2) \} \cap (\tau(e), \iota(e)) = \emptyset$ (as before). (Indeed, each edge $\{ v'_{1}, \ldots, v'_{n-2}, \beta(e_2), e \}$ must in fact be collapsible, in contrast with the situation covered in the previous case.)  It follows again that $c(v_{1}, \ldots, v_{n-2}, \alpha(e_{1}); \beta(e_{2})) = 1$. This proves (1).

Now we assume that $e_{1}$, $e_{2} \not \subseteq T$ and $\alpha(e_{1}) < \beta(e_{2})$.  Since $\{ v_{1}, \ldots, v_{n-2}, e_{1}, e_{2} \}$ is critical.  Because $\Gamma$ and $T$ satisfy Assumption \ref{assumption}, it must be that all of the vertices $\{ v_{1}, \ldots, v_{n-2} \}$ are blocked at the basepoint $\ast$.  The relative flow $f(v_{1}, \ldots, v_{n-2}, \alpha(e_{1}); \beta(e_{2}))$ is the edge-path obtained by starting at the vertex $\{ v_{1}, \ldots, v_{n-2}, \alpha(e_{1}), \beta(e_{2}) \}$ and repeatedly moving the vertex $\alpha(e_{1})$ until it is blocked at the basepoint.  All of the edges in this path are collapsible, so $c(v_{1}, \ldots, v_{n-2}, \alpha(e_{1}); \beta(e_{2})) = 1$.  This proves (2).

Now assume that $e_{1} \not \subseteq T$ and $e_{2} \subseteq T$.  Assume also that $\alpha(e_{1}) \wedge \tau(e_{2}) = \tau(e_{2})$, as in (3a).  In this case, the vertex $\{ v_{1}, \ldots, v_{n-2}, \alpha(e_{1}), \beta(e_{2}) \}$ consists of a collection of vertices blocked by $e_{2}$, a (possibly empty) collection of vertices blocked at $\ast$, and the vertices $\alpha(e_{1})$ and $\beta(e_{2})$.  The assumption $\alpha(e_1) \wedge \tau(e_2) = \tau(e_2)$ implies that $\alpha(e_{1})$ will eventually be blocked by $e_{2}$ (possibly after moving several times).  As $\alpha(e_1)$ is the only unblocked vertex, the relative flow $f(v_{1}, \ldots, v_{n-2}, \alpha(e_{1}); \beta(e_{2}))$ is the edge-path starting at $\{ v_{1}, \ldots, v_{n-2}, \alpha(e_{1}), \beta(e_{2}) \}$ which is determined by repeatedly moving $\alpha(e_{1})$ until it is blocked by $e_{2}$.  Each edge in this edge-path is collapsible, so $c(v_{1}, \ldots, v_{n-2}, \alpha(e_{1}); \beta(e_{2})) = 1$.

Assume that $e_{1} \not \subseteq T$, $e_{2} \subseteq T$, and $\alpha(e_{1}) < \tau(e_{2})$, as in (3b).  In this case, all of the vertices $\{ v_{1}, \ldots, v_{n-2} \}$ are either blocked by $e_{2}$, or blocked at the basepoint.  Since $\alpha(e_1) < \tau(e_2)$, the vertex $\alpha(e_1)$ can be moved repeatedly until it is blocked at the basepoint.  The relative flow $f(v_{1}, \ldots, v_{n-2}, \alpha(e_{1}); \beta(e_{2}))$ is the edge-path beginning at $\{ v_{1}, \ldots, v_{n-2}, \alpha(e_{1}), \beta(e_{2}) \}$ which can be obtained by moving $\alpha(e_{1})$ repeatedly until it is blocked at the basepoint.  All of the edges in this path are collapsible,  so $c(v_{1}, \ldots, v_{n-2}, \alpha(e_{1}); \beta(e_{2})) = 1$.  This prove (3).

Finally, we consider case (4):  $e_{1} \subseteq T$, $e_{2} \not \subseteq T$, and there is some vertex $v \in T$ such that 
$0 < d(v, \tau(e_{1})) < d(v, \beta(e_{2}))$.  
The vertices $\{ v_{1}, \ldots, v_{n-2} \}$ are all either blocked by $e_{1}$, or blocked at $\ast$.  It follows from our assumptions that each vertex in $\{ v_{1}, \ldots, v_{n-2}, \alpha(e_{1}) \}$ is less than $\beta(e_{2})$, so $f( v_{1}, \ldots, v_{n-2}, \alpha(e_{1}); \beta(e_{2}) )$ consists entirely of collapsible edges.  Therefore $c( v_{1}, \ldots, v_{n-2}, \alpha(e_{1}); \beta(e_{2}) ) = 1$, proving (4).

The main statement, that one of $c(v_{1}, \ldots, v_{n-2}, \alpha(e_{1}); \beta(e_{2}))$ or $c(v_{1}, \ldots, v_{n-2}, \beta(e_{2}); \alpha(e_{1}))$ is trivial, follows, because if $\{ v_{1}, \ldots, v_{n-2}, e_{1}, e_{2} \}$ fails to satisfy any of the conditions (1)-(4), then it must satisfy one of the conditions (1')-(4'), where (1')-(4') are the conditions that are be obtained from (1)-(4) by reversing the roles of $e_{1}$ and $e_{2}$ (and of $\alpha(e_{1})$ and $\beta(e_{2})$).
\end{proof}   

\begin{note}
In Proposition \ref{prop:cost}, certain of the costs, 
which are necessarily words in the alphabet of oriented critical $1$-cells, are described using both collapsible and 
critical $1$-cells.  This is a slight abuse of notation.  In the event that a cost contains collapsible $1$-cells among
its letters, then one should apply the collapsing rule from Subsection \ref{sec:fundgp} to remove each such $1$-cell
from the final form of the cost.
\end{note}

\begin{proposition} \label{prop:cost}
(Costs $n \geq 2$)  Assume that $c = \{ v_{1}, \ldots, v_{n-2}, e_{1}, e_{2} \}$ is a critical $2$-cell.  We frequently write $\mathcal{V}$ in place of $\{v_{1}, \ldots, v_{n-2} \}$.   Let $\alpha, \beta \in \{ \iota, \tau \}$.  Let $A = \tau(e_{1})$ and $B = \tau(e_2)$, and let $i$ be the direction from $A$ along $e_1$ and $j$ the direction from $B$ along $e_2$.

\begin{enumerate}
\item Assume $e_{1}, e_{2} \not\subset T$.  

\begin{enumerate}
\item If $\alpha(e_{1}) < \beta(e_{2})$, then
	$$c(\mathcal{V}, \alpha(e_{1}); \beta(e_{2})) = 1.$$
\item If $\alpha(e_{1}) > \beta(e_{2}) \neq \ast$, then
	$$ \smash{c(\mathcal{V}, \alpha(e_{1}); \beta(e_{2})) = C_{j}[ \delta_{i} + \delta_{j}],}$$
where $C = \alpha(e_{1}) \wedge \beta(e_{2})$, $i = d(C, \beta(e_{2}))$, and $j = d(C, \alpha(e_{1}))$.  
\item If $\beta(e_{2}) = \ast$, then
	$$c(\mathcal{V}, \alpha(e_{1}); \beta(e_{2})) = 1.$$
\end{enumerate}

\item Assume $e_1, e_2 \subset T$.  Then $c$ has the form $A_{i}[\vec{a}] + B_{j}[\vec{b}]$.

\begin{enumerate}
\item If $A \wedge B = A$, then 
	\begin{align*}
	c(\mathcal{V}, \iota(e_{1}); \beta(e_{2})) 
	&= \prod_{\ell = 0}^{|\vec{a}| - 1}  A_{d}[ (\vec{a} - \ell) + |\vec{b}|\delta_{d(A,B)}]; \\
	c(\mathcal{V}, \tau(e_{1}); \beta(e_{2})) 
	&=  \prod_{\ell = 0}^{|\vec{a}| - 2}  A_{d}[ ((\vec{a} - \delta_{i}) - \ell) + |\vec{b}|\delta_{d(A,B)}].
	\end{align*}
Here $d$ is the place of the first non-zero coordinate in $(\vec{a} - \delta_{i}) - \ell$ (in the first case) or in $\vec{a} - \ell$ (in the second case).

\item If $A \wedge B = C$ (where $C \not \in \{ A, B \}$), and $A > B$, then
	\begin{align*}
	c(\mathcal{V}, \alpha(e_{1}); \beta(e_{2})) 
	&= \prod_{\ell=0}^{|\vec{a}| - 1}  C_{d(C,A)} [(|\vec{a}| - \ell)\delta_{d(C,A)}+|\vec{b}| \delta_{d(C,B)}].
	\end{align*}
\item If $A \wedge B = B$ or if $A \wedge B = C$ and $A < B$, then
	$$c(\mathcal{V}, \alpha(e_{1}); \beta(e_{2})) = 1.$$

\end{enumerate}

\item Assume $e_1 \not\subset T$ and $e_2 \subset T$.  Then $c$ has the form $e_{1} + B_{j}[\vec{b}]$.

\begin{enumerate}
\item If $\alpha(e_{1}) \wedge B = C \neq B$ and $\alpha(e_{1}) > B$, then
	\begin{align*}
	c(\mathcal{V}, \alpha(e_{1}); \beta(e_{2})) 
	&= C_{d(C,\alpha(e_{1}))}[ |\vec{b}| \delta_{d(C,B)} + \delta_{d(C,\alpha(e_{1}))}].
	\end{align*}
\item If $\alpha(e_1) \wedge B = B$ or $\alpha(e_1) < B$, then
	$$c(\mathcal{V}, \alpha(e_{1}); \beta(e_{2})) = 1.$$
\end{enumerate}

\item Assume $e_1 \subset T$ and $e_2 \not\subset T$.  Then $c$ has the form $A_{i}[ \vec{a}] + e_{2}$.

\begin{enumerate}
\item If $A \wedge \beta(e_{2}) = A$, then
	\begin{align*}
	c(\mathcal{V}, \tau(e_{1}); \beta(e_{2})) 
	& =  \prod_{\ell = 0}^{|\vec{a}| - 2}  A_{d}[ ((\vec{a} - \delta_{i}) - \ell) + \delta_{d(A,\beta(e_{2}))}];  \\
	c(\mathcal{V}, \iota(e_{1}); \beta(e_{2})) 
	& = \prod_{\ell = 0}^{|\vec{a}| - 1}  A_{d}[ (\vec{a} - \ell) + \delta_{d(A,\beta(e_{2}))}].  
	\end{align*}
Here again $d$ is the place of the first non-zero coordinate in $(\vec{a} - \delta_{i}) - \ell$ (in the first case) or in $\vec{a} - \ell$ (in the second case).

\item If $A \wedge \beta(e_{2}) = C \neq A$ and $A > \beta(e_{2})$, then
	\begin{align*}
	c(\mathcal{V}, \alpha(e_{1}); \beta(e_{2})) 
	&= \prod_{\ell = 0}^{|\vec{a}| - 1} C_{d(C,A)}[ (|\vec{a}| - \ell)\delta_{d(C,A)} 
		+ \delta_{d(C, \beta(e_{2}))}].
	\end{align*} 

\item If $A \wedge \beta(e_2) = C \neq A$ and $A < \beta(e_2)$, then
	$$c(\mathcal{V}, \alpha(e_{1}); \beta(e_{2})) = 1.$$

\end{enumerate}
\end{enumerate}
\end{proposition}

\begin{proof}

\begin{enumerate}
\item Subcase (a) is covered by Proposition \ref{prop:vanishing}.  For Subcase (b), we assume first that $\alpha(e_{1}) > \beta(e_{2})$ and that neither $e_{1}$ nor $e_{2}$ contains the basepoint $\ast$.  It follows that $c = \{ \ast, 1, 2, \ldots, n-3, e_{1}, e_{2} \}$.  The flow $f( \mathcal{V}, \alpha(e_{1}); \beta(e_{2}))$ is therefore simply the edge-path $\{ \ast, 1, \ldots, n-3, [n-2, \alpha(e_{1})], \beta(e_{2}) \}$.  The edge-path $[n-2, \alpha(e_{1})]$ (in $\Gamma$) contains a unique edge $e$ such that $\tau(e) < \beta(e_{2}) < \iota(e)$.  The only edge in $f(\mathcal{V}, \alpha(e_{1}); \beta(e_{2}))$ contributing to $c(\mathcal{V}, \alpha(e_{1}); \beta(e_{2}))$ is therefore $\{ \ast, 1, \ldots, n-3, e, \beta(e_{2}) \}$, as all other edges in $f(\mathcal{V}, \alpha(e_{1}); \beta(e_{2}))$ satisfy the hypotheses of Lemma \ref{lemma:redundant}(2).  We note that $\tau(e) = \alpha(e_{1}) \wedge \beta(e_{2}) = C$.  Now we apply the flow to $\{ \ast, 1, \ldots, n-3, e, \beta(e_{2}) \}$; since $\tau(e) < \beta(e_{2}) < \iota(e)$ and all other vertices are blocked, we may apply Lemma \ref{lemma:redundant}(1) repeatedly, moving $\beta(e_{2})$ until it is blocked by $e$.  The resulting cell is clearly $C_{j}[ \delta_{i} + \delta_{j}]$, and therefore $c(\mathcal{V}, \alpha(e_{1}); \beta(e_{2})) = C_{j}[\delta_{i} + \delta_{j}]$, as claimed.

The proof is essentially the same if one of $e_{1}$ and $e_{2}$ contain the basepoint.  We omit the argument.  

Finally, for Subcase (c), we assume $\beta(e_{2}) = \ast$.  Then $f(\mathcal{V}, \alpha(e_{1}); \beta(e_{2})) = \{ \ast, 1, \ldots, n-2, [n-1, \alpha(e_{1})] \}$, which is made up entirely of collapsible edges.  Therefore $c(\mathcal{V}, \alpha(e_{1}); \beta(e_{2})) = 1$ in this case as well.

\item We first consider the cost $c(\mathcal{V}, \iota(e_{1}); \beta(e_{2}))$ in Subcase (a).  The flow 
$f(\mathcal{V}, \iota(e_{1}); \beta(e_{2}))$ has the initial vertex $\{ v_{1}, \ldots, v_{n-2}, \iota(e_{1}), \beta(e_{2}) \}$. The elements of this vertex (all of which are vertices in $\Gamma$) are of two types.  First, there are $|\vec{b}| - 1$ vertices blocked by $e_{2}$, and $\beta(e_{2})$ itself.  Second, there are $|\vec{a}|$ vertices clustered around the essential vertex $A$, where the number of such vertices in the direction $i$ from $A$ is the $i$th component of $\vec{a}$.  The flow $f(\mathcal{V}, \iota(e_{1}); \beta(e_{2}))$ is determined by successively moving each of the latter vertices back to the basepoint, beginning with the smallest-numbered.

We suppose, without loss of generality, that $v_{1}$ is the smallest of the vertices clustered at $A$.  Note that $e(v_{1})$ is the edge pointing in direction $d$ from $A = \tau(e(v_{1}))$.  The edge-path $\{ [\ast, v_{1}], v_{2}, \ldots, v_{n-2}, \iota(e_{1}); \beta(e_{2}) \}$ must consist entirely of collapsible edges, except possibly  the first edge $c' = \{ e(v_{1}), v_{2}, \ldots, v_{n-2}, \iota(e_{1}), \beta(e_{2})\}$, by 
Lemma \ref{lemma:redundant}(2).  We assume for the moment that $c'$ is not collapsible.  It follows that there are vertices of $c'$ contained in the interval $(A, v_1)$.  The only possibility is that the set of all such vertices is precisely the set of $|\vec{b}|$ vertices clustered around $B$.  If we apply the flow to the edge $c'$, we therefore arrive at the cell
	\begin{align}
	A_{d}[ \vec{a} + |\vec{b}|\delta_{d(A,B)}]. \label{bigimportant}  
	\end{align}
This cell is critical under our current assumptions (i.e., that $c'$ is not collapsible), because we must have $d(A,B) < d$. If $c'$ is collapsible, then we must have $d(A,B) \geq d$.  This implies that \eqref{bigimportant} is trivial.  It follows therefore that $M^{\infty}(c')$ is equal to \eqref{bigimportant} in any case.  This is the first factor in the cost.

We now continue, moving the second vertex from the cluster at $A$.  We can apply the same reasoning as before, but this time we replace the vector $\vec{a}$ in the above reasoning with $\vec{a} - 1$.  It follows that the second factor in the cost is
	$$A_{d}[(\vec{a} - 1) + |\vec{b}|\delta_{d(A,B)}],$$
where $d$ is as defined in the theorem.  (We note that $d$ depends on $\vec{a} -1$.)  One repeats this reasoning, eventually arriving at the formula in the statement of the theorem.

We now consider the cost $c(\mathcal{V}, \tau(e_{1}); \beta(e_{2}))$.  The computation of this cost is exactly like that of $c(\mathcal{V}, \iota(e_{1}); \beta(e_{2}))$, except that the smallest vertex $\tau(e_{1}) = A$ can always be moved freely to the basepoint, contributing nothing to the cost, since $\tau(e_{1})$ is smaller than all of the vertices clustered around $B$ (and $A$).  This leaves us to repeat the calculation of the previous paragraphs, but beginning with $\vec{a} - \delta_{i}$ rather than $\vec{a}$.  This establishes the desired identity.

We now consider Subcase (b).  Since $C \neq A,B$ and $A > B$, it follows that $0 < d(C,B) < d(C,A)$.  The initial vertex of the flow $f(\mathcal{V}, \alpha(e_{1}); \beta(e_{2}))$ consists of $|\vec{a}|$ vertices clustered around $A$ and $|\vec{b}|$ vertices clustered around $B$.  The flow $f(\mathcal{V}, \alpha(e_{1}); \beta(e_{2}))$ is the edge-path determined by moving the vertices at $A$ back to the basepoint, while holding those at $B$ fixed.  

We let $v_{1} \in \mathcal{V} \cup \{ \alpha(e_{1}) \}$ be the smallest of the vertices clustered at $A$.  We move $v_{1}$ back to the basepoint $\ast$.  All of the edges along the resulting edge-path (in $\ud{n}{\Gamma}$) are collapsible, contributing nothing to the cost, with the sole exception of $c' = \{ e, v_{2}, \ldots, v_{n-1}, \beta(e_{2}) \}$, where $e$ is the edge pointing in direction $d(C,A)$ from $C$.  Now we repeatedly apply Lemma \ref{lemma:redundant}(1) to $c'$, to arrive at
	\begin{align}
	C_{d(C,A)}[ |\vec{a}| \delta_{d(C,A)} + |\vec{b}|\delta_{d(C,B)}].
	\end{align}
This is the first factor in $c(\mathcal{V}, \alpha(e_{1}); \beta(e_{2}))$.  Now we repeat this argument, replacing $|\vec{a}|$ with $|\vec{a}| - 1$ in the above reasoning.  We eventually arrive at the cost in the statement of the Theorem.  

\item The cost in (3a) is essentially identical to the cost in (2b) (simply replace $|\vec{a}|$ with $1$), and it is computed in essentially the same way.  Subcase (3b) follows from Proposition \ref{prop:vanishing}(3).

\item The formulas for cost in (4a) and (4b) are the same as those in (2a) and (2b) (respectively) if we let $|\vec{b}| = 1$, and they are proved in a directly analogous way.  Subcase (4c) follows from Proposition \ref{prop:vanishing}(4).

\end{enumerate}
\end{proof}

\section{The Case of Two Strands}\label{sec:morse2pres}

We let $n=2$, and continue to assume that $\Gamma$ and $T$ satisfy Assumption \ref{assumption}.  
This means simply that $\Gamma$ is a simplicial graph, and that $T$ is a maximal tree in $\Gamma$ such that each vertex having degree $1$ in $T$ has degree at most $2$ in $\Gamma$.  We begin by describing the costs $c(v_{1}; v_{2})$.

\begin{proposition} \label{prop:nistwo}
(Costs in the case $n=2$)  Let $n=2$ and assume that $\Gamma$,  $T$ satisfy Assumption \ref{assumption}.  Assume that $v_{1}$ and $v_{2}$ are endpoints of deleted edges.  We let $A = v_{1} \wedge v_{2}$, and let $d(A, v_{1}) = i$ and $d(A, v_{2}) = j$.
\begin{enumerate}
\item If $v_{1} < v_{2}$ or if $v_{2} = \ast$, then $c(v_{1}; v_{2}) = 1$.
\item If $v_{1} > v_{2}$ and $v_{2} \neq \ast$, then $c(v_{1}; v_{2}) = A_{i}[\delta_{i} + \delta_{j}]$.
\end{enumerate}
\end{proposition}

\begin{proof}
This follows from Proposition \ref{prop:vanishing} and Proposition \ref{prop:cost}.
\end{proof}

\begin{figure}[!h] 
\begin{center}
\includegraphics{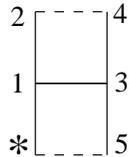}
\end{center}
\caption{The Morse presentation of the braid group $B_{2}\Gamma$ (for the $\Gamma$ pictured) has one relator.}
\label{fig:boxy}
\end{figure}  

\begin{example}
We let $n=2$, let $\Gamma$ be the graph pictured in Figure \ref{fig:boxy}, and let $T$ be the indicated maximal tree.

There is just one critical $2$-cell, $\{ [\ast, 5], [2,4] \}$.  Now note that $c(5;4) = B_{2}[1,1]$, $c(5;2) = A_{2}[1,1]$, and all other costs are trivial.  (Here we've written $c(5;4)$ rather than $c(5; \iota([2,4]))$ for the sake of brevity.)  If we let $\mathbf{e}_{1} = \{ \ast, [2,4] \}$ and $\mathbf{e}_{2} = \{ 1, [\ast, 5] \}$, then
	$$ B_{2}\Gamma \cong \langle A_{2}[1,1], B_{2}[1,1], \mathbf{e}_{1}, \mathbf{e}_{2} \mid 
	\mathbf{e}_{2}\mathbf{e}_{1}\mathbf{e}_{2}^{-1} 
	= B_{2}[1,1]\mathbf{e}_{1}(A_{2}[1,1])^{-1} \rangle.$$
It is clear therefore that $B_{2}\Gamma$ is a free group on three generators.
\end{example}

\begin{note} \label{note:important}
Theorem \ref{thm:biggie} and Proposition \ref{prop:nistwo} together give a general method for finding the relation corresponding to a given critical $2$-cell $c$ in the $n = 2$ case (and, therefore, for finding presentations for the groups $B_{2}\Gamma$).  The method can be sketched as follows.  We suppose that $c = \{ e_{1}, e_{2} \}$ and let $\mathbf{e}_{i}$ denote the critical $1$-cell consisting of the edge $e_{i}$ and a single vertex blocked at the basepoint. (That is, $\mathbf{e}_{i} = \{ \ast, e_{i} \}$ if $\ast \cap e_{i} = \emptyset$, and $\mathbf{e}_{i} = \{ 1, e_{i} \}$ otherwise.)

\begin{enumerate}
\item We begin by drawing an octagon and labelling four of the faces (as pictured below).
\begin{figure}[!h]
\begin{center}
\includegraphics{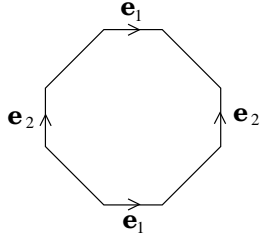}
\end{center}
\caption{The relators in a presentation for $B_{2}\Gamma$ can be represented by labelled octagons.  Here is one such octagon, with four of the eight sides labelled.}
\label{fig:octy}  
\end{figure}

\item Each of the four remaining faces will be labelled by a single critical $1$-cell (or by the trivial element).  This follows from Proposition \ref{prop:nistwo} and Theorem \ref{thm:biggie}.  We now determine the orientations on each of the slanted faces.  Consider first the orientation on the bottom left slanted face.  If $\iota(e_{1}) > \iota(e_{2})$, then the orientation on the bottom left face points up and to the left; if $\iota(e_{2}) > \iota(e_{1})$, then the orientation points down and to the right.  Thus, if $\iota(e_{i}) > \iota(e_{j})$, then the orientation points from the initial vertex of $\mathbf{e}_{i}$ to the initial vertex of $\mathbf{e}_{j}$.  

One follows exactly the same procedure at each slanted face:  the orientation points away from the larger vertex, and towards the smaller vertex (if we confuse, for the moment, the edges $\mathbf{e}_{\ell}$ and $e_{\ell}$ ($ \ell \in \{ i, j \}$)).

\item Finally, we label each corner with a critical $1$-cell (i.e., a generator).  Consider the bottom right corner, which connects the terminal vertex of $\mathbf{e}_{1}$ with the initial vertex of $\mathbf{e}_{2}$.   Let $A$ denote $\tau(e_{1}) \wedge \iota(e_{2})$.  Let $d(A, \tau(e_{1})) = i$ and $d(A, \iota(e_{2}))= j$.  If either $i$ or $j = 0$, which can happen only if $\tau(e_{1})$ or $\iota(e_{2}) = \ast$ (respectively), then we label the corner with $1$.  Otherwise, we label the corner with either $A_{j}[\delta_{i} + \delta_{j}]$ (if it is critical) or with $A_{i}[\delta_{i} + \delta_{j}]$ (if it is critical).  (Exactly one of these cells is critical, and the other is collapsible.)
\end{enumerate}
\end{note}

\begin{example} \label{ex:wheel} Consider $B_{2}\Gamma$, where $\Gamma$ is the graph depicted in Figure \ref{fig:wheel}.  We have chosen a subdivision of $\Gamma$ and an embedding of $\Gamma$ into the plane.

\begin{figure}[!h]
\begin{center}
\includegraphics{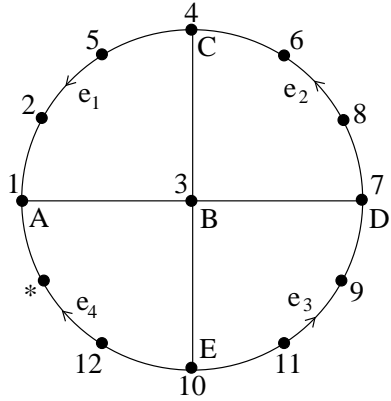}

\caption{Here is a simple graph $\Gamma$.  We compute a presentation for $B_{2}\Gamma$.}
\label{fig:wheel}
\end{center}
\end{figure}

\begin{figure}[!h] 
\begin{center}
\includegraphics{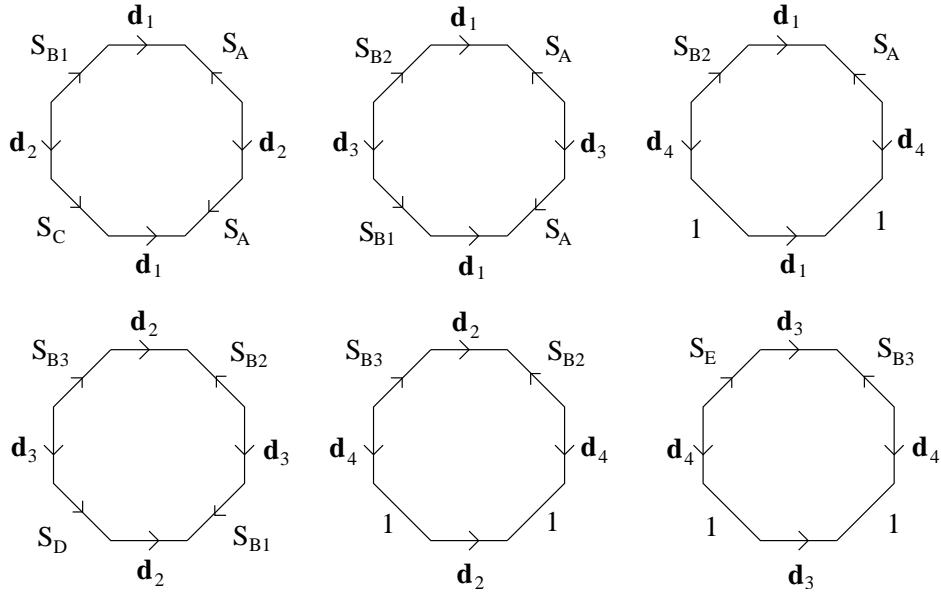}

\caption{These are the relators for the $2$-strand braid group on the graph in Figure \ref{fig:wheel}.}
\label{fig:examplerelators}
\end{center}
\end{figure}



We follow our usual convention for labelling essential vertices by letters (see Subsection \ref{subsec:prelim}), so $A=1$, $B=3$, $C=4$, $D=7$, and $E=10$.  There are eleven critical $1$-cells: $\mathbf{d}_{i} = \{ \ast, e_{i} \}$ ($i = 1,2,3$), $\mathbf{d}_{4} = \{ 1, e_{4} \}$, $S_{B1} = B_{2}[1,1,0]$, 
$S_{B2} = B_{3}[1,0,1]$, $S_{B3} = B_{3}[0,1,1]$, and $S_{X} = X_{2}[1,1]$, for $X \in \{ A, C, D \}$. There are six critical $2$-cells $\{ e_{i}, e_{j} \}$ ($1 \leq i < j \leq 4$).

We can follow the procedure sketched in Note \ref{note:important} in order to find the relations.  They are pictured in Figure \ref{fig:examplerelators}.

We can eliminate by Tietze transformations the generators $E_{2}[1,1]$, $D_{2}[1,1]$, $C_{2}[1,1]$, $B_{2}[1,1,0]$, $A_{2}[1,1]$, and $B_{3}[1,0,1]$ (in that order).  The resulting presentation shows that $B_{2}\Gamma$ is a free group on five generators:
	$$ B_{2} \Gamma \cong \langle \mathbf{d}_{1}, \mathbf{d}_{2}, \mathbf{d}_{3}, \mathbf{d}_{4}, 
	B_{3}[0,1,1] \rangle.$$
\end{example}

\begin{proposition}
If $\Gamma$ is a graph such that there exists a vertex $v_0$ in $\Gamma$ which is on every simple loop in $\Gamma$,  then $B_2\Gamma$ is free.
\end{proposition}

\begin{proof}
We want to show that one can choose a maximal tree $T \subseteq \Gamma$ such that every deleted edge touches $v_{0}$. Since critical $2$-cells in $\ud{2}{\Gamma}$ are in one-to-one correspondence  with unordered pairs of disjoint deleted edges, the existence of such a tree shows that there are no critical $2$-cells, and therefore that the Morse presentation of $\ud{2}{\Gamma}$ defines a free group.

We now find a tree $T$ with the desired property.  The assumptions imply that $\Gamma- \{ v_{0} \}$ consists of a disjoint union of trees $T_{1}, \ldots, T_{n}$.  We fix a tree $T_{i}$.  Certain of the edges of $T_{i}$ are half-open, since their closures contain $v_{0}$.  We let $e_{i1}, \ldots, e_{ik}$ denote these half-open edges.  We let
	$$ \widehat{T}_{i} = T_{i} - \left( \bigcup_{j=2}^{k} \mathring{e}_{ij} \right),$$
where $\mathring{e}$ denotes the interior of $e$.  Note that $\widehat{T}_i$ contains $\mathring{e}_{i1}$, and so $\widehat{T}_i$ is a tree with exactly 1 leaf at $v_0$.  The union $T = {v_0} \sqcup \widehat{T}_{1} \sqcup \ldots \sqcup \widehat{T}_{n}$ is thus a maximal tree in $\Gamma$ with the property that all edges $e \not \subseteq T$ touch $v_{0}$.
\end{proof}

\begin{theorem}\label{thm:disjointloops}
Let $\Gamma$ be a graph in which every pair of simple loops is disjoint.  The group $B_{2}\Gamma$ has a  Morse presentation such that all relators are commutators corresponding to pairs of simple loops.
\end{theorem}

\begin{proof}
We consider a copy of $\Gamma$ (with maximal tree $T$) satisfying Assumption \ref{assumption}.  Let $c$ be a critical $2$-cell in $\ud{2}{\Gamma}$.  It follows that $c= \{ e_{1}, e_{2} \}$ for some deleted edges $e_{1}, e_{2} \subseteq \Gamma$.

We first note that $\gamma_{i} = [ \tau(e_{i}), \iota(e_{i})] \cup e_{i}$ is a simple loop for $i=1,2$, so, in particular, $\gamma_{1} \cap \gamma_{2} = \emptyset$.  We let $p_{i}$ denote the smallest vertex on $\gamma_{i}$ for $i=1,2$.  We assume, without loss of generality, that $p_{1} < p_{2}$.

The relation determined by $\{e_{1}, e_{2} \}$ takes the following form, by Theorem \ref{thm:biggie}:
	\begin{align}
	w_{\iota,\tau}\mathbf{e}_{1}w_{\tau,\tau}^{-1}\mathbf{e}_{2}^{-1}w_{\tau,\iota}\mathbf{e}_{1}^{-1}
	w_{\iota,\iota}^{-1}\mathbf{e}_{2},
	\end{align}
where 
	\begin{align}
	w_{\alpha,\beta} = c( \alpha(e_{1});\qquad \beta(e_{2}))^{-1}c(\beta(e_{2}); \qquad \alpha(e_{1})),
	\end{align}
for $\alpha, \beta \in \{ \iota, \tau \}$.  It will follow that the above relation is a commutator relation if we show that $\iota(e_{2})$ and $\tau(e_{2})$ may both be replaced by $p_{2}$ in all of the above costs, for then  $w_{\iota,\tau} = w_{\iota, \iota}$ and $w_{\tau, \iota} = w_{\tau, \tau}$, so the above relation becomes $$[w_{\iota,\iota} \mathbf{e}_{1} w_{\tau, \tau}^{-1}, \mathbf{e}_{2}^{-1}].$$

We first show that $c(\iota(e_{2}); \iota(e_{1})) = c(p_{2}; \iota(e_{1})) = c(\tau(e_{2}); \iota(e_{1}))$.  There are two cases: either $[\ast, p_{2}] \cap \gamma_{1} = \emptyset$ or $[\ast, p_{2}] \cap \gamma_{1} \neq \emptyset$. If $[\ast, p_{2}] \cap \gamma_{1} = \emptyset$, then all of the vertices on $\gamma_{2}$ are greater than all of the vertices on $\gamma_{1}$.  The flows $f(\iota(e_{2}); \iota(e_{1}))$ and $f(\tau(e_{2}); \iota(e_{1}))$ both contain the subpath $f(p_{2}; \iota(e_{1}))$.  This subpath contains the only edge making a contribution to either cost, namely $\{ \iota(e_{1}), e \}$, where $e$ is the edge pointing in direction $d(p_{1} \wedge p_{2}, p_{2})$ from $p_{1} \wedge p_{2}$.  Therefore, $c(\iota(e_{2}); \iota(e_{1})) = c( \tau(e_{2}); \iota(e_{1}))$.  If $[\ast, p_{2}] \cap \gamma_{1} \neq \emptyset$, then the argument is similar, except that the only contribution comes from the edge $\{ \iota(e_{1}), e \}$, where $e$ is the edge pointing in direction $d(\iota(e_{1}) \wedge p_{2}, p_{2})$ from $\iota(e_{1}) \wedge p_{2}$.  The edge $\{ \iota(e_{1}), e \}$ is still common to the edge-paths $f(\iota(e_{2}); \iota(e_{1}))$ and $f(\tau(e_{2}); \iota(e_{1}))$, and therefore $c(\iota(e_{2}); \iota(e_{1})) = c(\tau(e_{2}); \iota(e_{1}))$.

Now we show that $c(\iota(e_{1}); \iota(e_{2})) = c(\iota(e_{1}); p_{2}) = c(\iota(e_{1}); \tau(e_{2}))$.  If 
$\iota(e_{1}) < p_{2}$ (and therefore $\iota(e_{1}) < \iota(e_{2})$ and $\iota(e_{1}) < \tau(e_{2})$), then all of these costs are trivial, so we may assume that $\iota(e_{1}) > p_{2}$.  Now the path $[ \ast, \iota(e_{1})] = [\ast, p_{1}] \cup [p_{1}, \iota(e_{1})]$ doesn't contain the point $p_{2}$, so there will be a unique edge $e$ in the edge-path $[\ast,\iota(e_{1})]$ such that $\tau(e) < p_{2} < \iota(e)$.

We claim that $\iota(e_{2}), \tau(e_{2}) \in (p_{2}, \iota(e))$.  Indeed, the inequality $\tau(e) < p_{2} < \iota(e)$ implies that $0 < d(\tau(e), p_{2}) < d(\tau(e), \iota(e)) = d(\tau(e), \iota(e_{1}))$.  It is clear that $d(\tau(e), p_{2}) = d(\tau(e), \iota(e_{2})) = d(\tau(e), \tau(e_{2}))$, so $d(\tau(e), \iota(e_{2})), d(\tau(e), \tau(e_{2})) < d(\tau(e), \iota(e_{1}))$.  It follows directly that $\iota(e_{2}), \tau(e_{2}) < \iota(e_{1})$.  The inequalities $p_{2} < \tau(e_{2}), \iota(e_{2})$ follow from the definition of $p_{2}$.  This proves the claim.

The flows $f(\iota(e_{1}); \iota(e_{2}))$ and $f(\iota(e_{1}); \tau(e_{2}))$ contain the edges $\{ e, \iota(e_{2})\}$ and $\{ e, \tau(e_{2}) \}$ (respectively; $e$ as above), and these are the only edges that will contribute to the costs.  Both of these edges flow to $\{ v, e \}$ by Lemma \ref{lemma:redundant}(1), where $v$ is the vertex adjacent to $\tau(e)$ satisfying $d(\tau(e), v) = d(\tau(e), p_{2})$.  It follows that $c(\iota(e_{1}); \iota(e_{2})) = c( \iota(e_{1}); \tau(e_{2}))$.

The preceding argument shows that $\tau(e_{2})$ and $\iota(e_{2})$ may be replaced by $p_{2}$ in the costs  $c(\tau(e_{2}); \iota(e_{1}))$, $c(\iota(e_{2}); \iota(e_{1}))$, $c(\iota(e_{1}); \tau(e_{2}))$, and $c(\iota(e_{1}); \iota(e_{2}))$. Essentially the same argument show that $\tau(e_{2})$ and $\iota(e_{2})$ may be replaced by $p_{2}$ in the costs involving $\tau(e_{1})$ (rather than $\iota(e_{1})$).  

To finish the proof, we note that unordered pairs of disjoint simple loops are in bijective correspondence with unordered pairs of distinct deleted edges; i.e., with critical $2$-cells, and therefore with relators in the Morse presentation.
\end{proof}

The following conjecture is a large generalization of Theorem \ref{thm:disjointloops}:

\begin{conjecture}
Let $\Gamma$ be a planar graph.  Then there exists a presentation for $B_2\Gamma$ such that the relators are all commutators corresponding to pairs of disjoint simple loops.
\end{conjecture}

An original goal for this paper was to prove this conjecture.  In a large number of examples, the authors have verified the result.  In general, though, a proof is elusive.

\section{A Class of Examples}
\label{sec:threeandmore}

Let $I=[0,1]$.  The wedge sum of $I$ with the circle $S^{1}$ along the basepoint $0 \in I$ is called a \emph{balloon}.  A \emph{balloon graph} is a wedge sum of a finite collection of balloons with a line segment, where all basepoints in the wedge sum are of degree one.  A balloon graph is uniquely determined up to homeomorphism by the number of its balloons.

Fix a balloon graph $\Gamma$ with $m$ balloons.  We will compute presentations of $B_{n}\Gamma$ for $n=2,3$.

\subsection{The case $n=2$}

We choose a copy of $\Gamma$ satisfying Assumption \ref{assumption}.  We embed $\Gamma$ in the plane as in Figure \ref{fig:balloons2}.

\begin{figure}[!h]
\begin{center}
\includegraphics{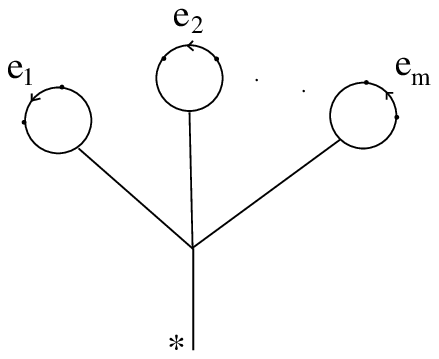}
\end{center}
\caption{This is a picture of a balloon graph with $m$ balloons.  The edges $e_{1}$, $\ldots$, $e_{m}$ are the deleted edges.}
\label{fig:balloons2}
\end{figure}

We will let $A$ denote the lower essential vertex of degree $m+1$.  We will let $B_{i}$ denote the $i$th vertex of degree three from the left (so $B_{i}$ is part of the circle containing the edge $e_{i}$).

The critical $1$-cells in $\ud{n}{\Gamma}$ are of three types:

\begin{enumerate}
\item There are $m$ critical $1$-cells $\{ \ast, e_{i} \}$, where $i \in \{ 1, \ldots, m \}$.  We let $\mathbf{e}_{i}$ denote the critical $1$-cell $\{ \ast, e_{i} \}$.

\item A total of $m$ critical $1$-cells are described by the vector notation $(B_{i})_{2}[1,1]$ ($i \in \{ 1, \ldots, m \}$).  Since a critical $1$-cell of this type is completely determined by the essential vertex $B_{i}$, we write $\mathbf{B}_{i}$ in place of $(B_{i})_{2}[1,1]$.  

\item A total of $\left( \begin{matrix} m \\ 2 \end{matrix} \right)$ critical $1$-cells are described by the vector notation $A_{j}[ \delta_{i} + \delta_{j}]$ ($ 1 \leq i < j \leq m$).  We denote these $\mathbf{A}_{i,j}$.
\end{enumerate}

There are $\left( \begin{matrix} m \\ 2 \end{matrix} \right)$ critical $2$-cells, all of the form $\{ e_{i}, e_{j} \}$ ($1 \leq i < j \leq m$).  It is straightforward to follow the procedure from Note \ref{note:important} and arrive at the following relations.
	$$ [ \mathbf{e}_{j}, \mathbf{A}_{i,j}\mathbf{e}_{i}\mathbf{A}_{i,j}^{-1} ] \quad (1 \leq i < j \leq m).$$
This gives us the following presentation for $B_{2}\Gamma$:
	$$ \langle \mathbf{e}_{i}, \mathbf{B}_{i} (i \in \{ 1, \ldots, m \}), \mathbf{A}_{i,j} (1 \leq i < j \leq m)
	\mid [ \mathbf{e}_{j}, \mathbf{A}_{i,j}\mathbf{e}_{i}\mathbf{A}_{i,j}^{-1} ] \quad (1 \leq i < j \leq m)
	\rangle.$$

\subsection{The case $n=3$}

We choose a version of $\Gamma$ satisfying the Assumption \ref{assumption}.  We carry over our convention of labelling essential vertices from the case $n=2$.

The critical $1$-cells in $\ud{3}{\Gamma}$ are of three types:

\begin{enumerate}
\item There are $m$ critical $1$-cells $\{ \ast, 1, e_{i} \}$, where $i \in \{ 1, \ldots, m \}$.  We let $\mathbf{e}_{i}$ denote the critical $1$-cell $\{ \ast, 1, e_{i} \}$.  

\item A total of $3m$ critical $1$-cells are described by the vector notation $(B_{i})_{2}[a,b]$ ($a,b \geq 1$; $a+b \leq 3$).  A critical $1$-cell of this type is completely determined by $B_{i}$, $a$, and $b$, so we write $\mathbf{B}_{i}[a,b]$ instead.

\item A total of $3 \left( \begin{matrix} m \\ 2 \end{matrix} \right)$ critical $1$-cells are determined by the
vector notation $A_{j}[ a\delta_{i} + b\delta_{j}]$ ($a,b \geq 1$; $a+b \leq 3$; $1 \leq i < j \leq m$).  An
additional $2\left( \begin{matrix} m \\ 3 \end{matrix} \right)$ critical $1$-cells are determined by 
$A_{\ell}[\delta_{i} + \delta_{j} + \delta_{k}]$ ($1 \leq i < j < k \leq m$; $\ell \in \{ j, k \}$).  This makes a
total of $3 \left( \begin{matrix} m \\ 2 \end{matrix} \right) + 2 \left( \begin{matrix} m \\ 3 \end{matrix} \right)$
critical $1$-cells of this type.
\end{enumerate}
We therefore have a total of $4m + 3\left( \begin{matrix} m \\ 2 \end{matrix} \right) + 
2 \left( \begin{matrix} m \\ 3 \end{matrix} \right)$ critical $1$-cells.

There are two types of critical $2$-cells:

\begin{enumerate}
\item There are $\left( \begin{matrix} m \\ 2 \end{matrix} \right)$ critical $2$-cells of the
form $\{ \ast, e_{i}, e_{j} \}$ ($1 \leq i < j \leq m$).

\item We can form the other critical $2$-cells by making a choice of a deleted edge $e_{i}$ ($1 \leq i \leq m$) and a choice of a critical subconfiguration $(B_{i})_{2}[1,1]$ $(1 \leq i \leq m$) or $A_{j}[ \delta_{i} + \delta_{j}]$ ($1 \leq i < j \leq m$).  This makes $m \left( m + \left( \begin{matrix} m \\ 2 \end{matrix} \right) \right)$  possible critical $2$-cells of this kind.
\end{enumerate}
It follows that there are $m^{2} + (m+1) \left( \begin{matrix} m \\ 2 \end{matrix} \right)$ critical $2$-cells in all.

The relators take several different forms.  We simply enumerate the possible cases, describe the relators, and leave the verifications to the interested reader.

\begin{enumerate}
\item $(B_{j})_{2}[1,1] + e_{i}$ ($1 \leq i < j \leq m$):  The relator is:
	$$\smash{[ \mathbf{B}_{j}[1,1], A_{j}[\delta_{i} + 2\delta_{j}]A_{j}[ \delta_{i} + \delta_{j}] \mathbf{e}_{i}	A_{j}[ \delta_{i} + \delta_{j}]^{-1} A_{j}[ \delta_{i} + 2 \delta_{j}]^{-1}]}.$$

\item $(B_{i})_{2}[1,1] + e_{i}$ ($1 \leq i \leq m$):  
	$$\mathbf{e}_{i} \mathbf{B}_{i}[1,1]^{-1}\mathbf{B}_{i}[2,1]\mathbf{e}_{i}^{-1}\mathbf{B}_{i}[1,2]^{-1}.
	$$

\item $(B_{i})_{2}[1,1] + e_{j}$ ($1 \leq i < j \leq m$):
	$$[ \mathbf{e}_{j}, A_{j}[2\delta_{i} + \delta_{j}] \mathbf{B}_{i}[1,1] A_{j}[2 \delta_{i} + \delta_{j}]^{-1}].
	$$

\item $A_{j}[ \delta_{i} + \delta_{j}] + e_{i}$ ($1 \leq i < j \leq m$):
	$$ [ A_{j}[\delta_{i} + \delta_{j}]^{-1}A_{j}[2\delta_{i}+\delta_{j}], \mathbf{e}_{i}].$$

\item $A_{k}[ \delta_{j} + \delta_{k} ] + e_{i}$ ($1 \leq i < j < k \leq m$):
	$$[ A_{k}[\delta_{i} + \delta_{k}]^{-1}A_{j}[\delta_{i} + \delta_{j} + \delta_{k}]^{-1}A_{k}[\delta_{i} + 
	\delta_{j} + \delta_{k}] A_{j}[\delta_{i} + \delta_{j}], \mathbf{e}_{i}].$$

\item $A_{k}[ \delta_{i} + \delta_{k}] + e_{j}$ ($1 \leq i < j < k \leq m$):
	$$[A_{k}[\delta_{j} + \delta_{k}]^{-1}A_{k}[\delta_{i} + \delta_{j} + \delta_{k}], \mathbf{e}_{j}].$$

\item $A_{j}[ \delta_{i} + \delta_{j}] + e_{k}$ ($1 \leq i < j < k \leq m$):
	$$ [\mathbf{e}_{k}, A_{j}[\delta_{i} + \delta_{j} + \delta_{k}]].$$

\item $A_{j}[ \delta_{i} + \delta_{j}] + e_{j}$ ($1 \leq i < j \leq m$):
	$$ [\mathbf{e}_{j}, A_{j}[\delta_{i} + 2\delta_{j}]].$$

\item $\mathbf{e}_{i} + \mathbf{e}_{j}$ ($1 \leq i < j \leq m$):
	$$ [ \mathbf{e}_{j}, A_{j}[\delta_{i} + \delta_{j}]\mathbf{e}_{i}A_{j}[\delta_{i} + \delta_{j}]].$$
\end{enumerate}
Finally, we notice that all of the above relators are commutators, with the exception of those from (2).  We can eliminate each of the latter relations by Tietze transformations, and (at the same time) eliminate the generators $\mathbf{B}_{i}[1,1]$ ($1 \leq i \leq m$).  The resulting presentation for $B_{3}\Gamma$ has $3m + 3\left( \begin{matrix} m \\ 2 \end{matrix} \right) + 2 \left( \begin{matrix} m \\ 3 \end{matrix} \right)$ generators and $m^{2} - m + (m+1) \left( \begin{matrix} m \\ 2 \end{matrix} \right)$ relators, all of which are commutators.

\bibliography{refs-FS2}
\nocite{*}
\bibliographystyle{plain}

\end{document}